\documentclass[reqno]{amsart}
\usepackage{}
\usepackage{stmaryrd}
\usepackage{mathrsfs}
\usepackage{cases}
\usepackage{amsfonts}
\usepackage{amssymb}
\usepackage{amsmath}
\newtheorem{theorem}{Theorem}[section]
\newtheorem{lemma}[theorem]{Lemma}
\newtheorem{proposition}[theorem]{Proposition}
\newtheorem{corollary}[theorem]{Corollary}
\theoremstyle{definition}
\newtheorem{definition}[theorem]{Definition}

\newtheorem{remark}[theorem]{Remark}
\numberwithin{equation}{section}

\newcommand{\blankbox}[2]

\begin{document}
\title [Weighted modulation and Wiener amalgam spaces]
{CHARACTERIZATIONS OF SOME PROPERTIES ON WEIGHTED MODULATION AND WIENER AMALGAM SPACES }
\author{WEICHAO GUO}
\address{School of Mathematical Sciences, Xiamen University,
Xiamen, 361005, P.R.China} \email{weichaoguomath@gmail.com}
\author{jiecheng chen}
\address{Department of Mathematics, Zhejiang Normal University,
Jinhua, 321004, P.R.China} \email{jcchen@zjnu.edu.cn}
\author{DASHAN FAN}
\address{Department of Mathematics, University of Wisconsin-Milwaukee, Milwaukee, WI 53201, USA}
\email{fan@uwm.edu}
\author{GUOPING ZHAO}
\address{School of Applied Mathematics, Xiamen University of Technology,
Xiamen, 361024, P.R.China} \email{guopingzhaomath@gmail.com}


\begin{abstract}
In this paper, some properties on weighted modulation and Wiener amalgam spaces are characterized by the corresponding properties on weighted Lebesgue spaces.
As applications, sharp conditions for product inequalities, convolution inequalities and embedding on weighted modulation and Wiener amalgam spaces are obtained.
These applications improve and extend many known results.
\end{abstract}

\maketitle


\section{Introduction}
The study of modulation space, which originated by Feichtinger \cite%
{Feichtinger} about 30 years ago, has over time been transformed into a rich
and multifaceted theory, providing basic insights into such topics as
harmonic analysis, time-frequency analysis and partial differential
equation. Nowadays, the theory has played more and more notable roles. Among
numerous references, one can see \cite{Feichtinger_Survey} for the
historical perspectives and background on the motivations which led to the invention of the modulation spaces,
see \cite{Cordero_Nicola_Sharpness,
Galperin_Samarah, Sugimoto_Tomita, Toft_weighted modulation space,
Toft_sharp convolution, Tribel_modulation space} for understanding many
characterizations and fundamental properties of the modulation space, see
\cite{AKKL_JFA_2007,AKC_JOT_2005,Chen_Fan_Sun, Feichtinger_Narimani} for the study of relevant
operators on modulation space, and see \cite%
{Cordero_Nicola_JDE,WZG_JFA_2006,Wang_book,WH_JDE_2007} for the study of
nonlinear evolution equations related to the modulation space.

Let $\mathscr {S}:= \mathscr {S}(\mathbb{R}^{n})$ be the space of all Schwartz functions
and $\mathscr {S}':=\mathscr {S}'(\mathbb{R}^{n})$ be the space of all tempered distributions.
We define the Fourier transform $\mathscr{F}f$ and the inverse Fourier transform $\mathscr{F}^{-1}f$ of $f\in \mathscr{S}(\mathbb{R}^{n})$ by
$$
\mathscr{F}f(\xi)=\hat{f}(\xi)=\int_{\mathbb{R}^{n}}f(x)e^{-2\pi ix\cdot \xi}dx
~~
,
~~
\mathscr{F}^{-1}f(x)=\check{f}(x)=\int_{\mathbb{R}^{n}}f(\xi)e^{2\pi ix\cdot \xi}d\xi.
$$

The translation operator is defined as $T_{x_0}f(x)=f(x-x_0)$ and
the modulation operator is defined as $M_{\xi}f(x)=e^{2\pi i\xi \cdot x}f(x)$, for $x, x_0, \xi\in\mathbb{R}^n$.
Fixed a nonzero function $\phi\in \mathscr{S}$, the short-time Fourier
transform of $f\in \mathscr{S}'$ with respect to the window $\phi$ is given by
\begin{equation}
V_{\phi}f(x,\xi)=\langle f,M_{\xi}T_x\phi\rangle=\int_{\mathbb{R}^n}f(y)\overline{\phi(y-x)}e^{-2\pi iy\cdot \xi}dy.
\end{equation}

We use $L_{x; p}(\mathbb{R}^n)$ to denote the Banach space ($p\geq 1$) or Quasi-Banach space ($0<p<1$) of measurable functions $f:\mathbb{R}^n \rightarrow \mathbb{C}$, whose norms
\begin{equation}
\|f\|_{L_{x;p}(\mathbb{R}^n)}:= \left(\int_{\mathbb{R}^n}|f(x)|^{p}dx\right) ^{{1}/{p}}
\end{equation}
are finite, with the usual modification when $p=\infty$.
In many cases we will abbreviate $L_{x; p}(\mathbb{R}^n)$ as $L_{x; p}$ or even $L_p$, when there is no chance of confusion.

Let $\mathfrak{M}(x,\xi)$ be a non-negative function on $\mathbb{R}^{2n}$, and $\phi\in \mathscr{S}$ be a fixed window.
We define the norm (or quasi-norm)
\begin{equation}
\begin{split}
\|f\|_{\mathcal {M}_{p,q}^{\mathfrak{M}}}&=\big\|\|V_{\phi}f(x,\xi)\mathfrak{M}(x,\xi)\|_{L_{x;p}}\big\|_{L_{\xi; q}}
\\&
=\left(\int_{\mathbb{R}^n}\left(\int_{\mathbb{R}^n}|V_{\phi}f(x,\xi)\mathfrak{M}(x,\xi)|^{p}dx\right)^{{q}/{p}}d\xi\right)^{{1}/{q}}
\end{split}
\end{equation}
with the usual modification when $p=\infty$ or $q=\infty$.
Then, the
modulation space \ $\mathcal{M}_{p,q}^{\mathfrak{M}}$ \ is\ the set of all $\ f\in \mathscr {S}'(\mathbb{R}^n)$ \ satisfying \ $\left\Vert f\right\Vert _{\mathcal{M}_{p,q}^{\mathfrak{M}}}<\infty $.

From the definition, we see that
the modulation space is defined by measuring the integrability of $V_{\phi}f(x,\xi)$ in some suitable mixed-norm spaces on $\mathbb{R}^{2n}$ (the time-frequency plane).
Moreover, weights $\mathfrak{M}(x,\xi)$ on the time-frequency plane can be used to draw a more accurate portrait about the global integral properties
of the short-time Fourier transform, one can see \cite{Grochnig_weight} for a comprehensive discussion of weights in time-frequency analysis.
Also, we want to remind the reader that
the initial definition of $\mathcal {M}_{p,q}^{\mathfrak{M}}$ is for $1\leq p, q\leq \infty$,
while the reader can see \cite{Galperin_Samarah, Kobayashi, Tribel_modulation space}
for the definition of $\mathcal{M}_{p,q}^{\mathfrak{M}}$ on the full range $0<p,q\leq \infty $.
In this paper, we will adopt the definition mentioned in \cite{Galperin_Samarah}, which is consistent with the original definition on $1\leq p, q \leq \infty$.

Moderate weights occur in the definition of general modulation space.
In fact, the moderateness of weights appears quite naturally for the convolution estimates (see \cite{Grochnig}), which is the basic tool for studying time-frequency analysis.
More formally, for a weigh function $v$, a non-negative function $\omega$ defined on $\mathbb{R}^n$ is called $v$-moderate if
\begin{equation}
\omega(x+y)\leq Cv(x)\omega(y)
\end{equation}
for any $x,y\in \mathbb{R}^n$, where \ $C$ \ is a constant independent of \ $x,y\in \mathbb{R}^{n}$.

In this paper, we will consider the weights of polynomial growth.
We use the notation $\mathscr{P}(\mathbb{R}^n)$ to denote the cone of all non-negative functions $\omega$ which are
$v-$moderate, where $v$ is a polynomial on $\mathbb{R}^n$.
In addition, our main concern is the weights of separation of variables.
More precisely, we study the weighted modulation space $\mathcal {M}_{p,q}^{\mathfrak{M}}$ with $\mathfrak{M}(x,\xi)=\omega(x) m(\xi)$.
This separation property of weights roughly makes the behaviors of the weighted modulation spaces more close to the behaviors of corresponding weighted Lebesgue spaces.

Let $0<p,q\leq \infty$, and $\omega$ be a weight function. The function space $L_{x;p,\omega}$ consists of all measurable functions $f$ such that
\begin{numcases}{\|f\|_{L_{x;p,\omega}}=}
     \left(\int_{\mathbb{R}^n}|f(x)\omega(x)|^p dx\right)^{1/p}, &$p<\infty$   \\
     ess\sup_{x\in \mathbb{R}^n}|f(x)\omega(x)|,  &$p=\infty$
\end{numcases}
is finite. We write it as $L_{p,\omega}$ if there is no confusion. Also, we write $L_{p,s}$ for the case $\omega(x)=(1+ |x|^2)^{s/2}$
and $L_p=L_{p,0}$.

Now, we give the definition of weighted modulation space $\mathcal {M}_{p,q}^{m,\omega}$.

\begin{definition}\label{Definition, modulation space, continuous form}
Let $0<p,q\leq \infty$, $m, \omega \in \mathscr{P}(\mathbb{R}^n)$.
Given a window function $\phi\in \mathscr{S}\backslash\{0\}$, the (weighted) modulation space $\mathcal {M}_{p,q}^{m,\omega}$ consists
of all $f\in \mathscr{S}'(\mathbb{R}^n)$ such that the norm
\begin{equation}
\begin{split}
\|f\|_{\mathcal {M}_{p,q}^{m,\omega}}&=\big\|\|V_{\phi}f(x,\xi)\|_{L_{x;p,\omega}}\big\|_{L_{\xi;q,m}}
\\&
=\left(\int_{\mathbb{R}^n}\left(\int_{\mathbb{R}^n}|V_{\phi}f(x,\xi)\omega(x)|^{p} dx\right)^{{q}/{p}}|m(\xi)|^qd\xi\right)^{{1}/{q}}
\end{split}
\end{equation}
is finite, with the usual modification when $p=\infty$ or $q=\infty$.
In addition, we write $\mathcal {M}_{p,q}^{s,t}:= \mathcal {M}_{p,q}^{m,\omega}$ for the case $\omega=\langle x\rangle^t$ and $m(\xi)=\langle \xi\rangle^s$,
where $\langle x\rangle :=(1+|x|^{2})^{{1}/{2}}$.
We also write $\mathcal {M}_{p,q}^s:=\mathcal {M}_{p,q}^{s,0}$ and $\mathcal {M}_{p,q}:=\mathcal {M}_{p,q}^{0,0}$.
\end{definition}

The above definition of \ $\mathcal{M}_{p,q}^{m,\omega }$ \  is independent of the choice of window function $\phi$.
The reader may see this fact in \cite{Grochnig} \ for the case $(p,q)\in\lbrack 1,\infty ]^{2}$,
and in \cite{Galperin_Samarah} for the case $(p,q)\in (0,\infty ]^{2}\backslash\lbrack 1,\infty ]^{2}$.

Next, we introduce the Wiener amalgam space $W_{p,q}^{ m,\omega}$
corresponding to the space $\mathcal{M}_{p,q}^{m,\omega }$.

\begin{definition}
Let $0<p,q\leq \infty$, $m, \omega \in \mathscr{P}(\mathbb{R}^n)$.
Given a window function $\phi\in \mathscr{S}\backslash\{0\}$, the (weighted) Wiener amalgam space $W_{p,q}^{m,\omega}$ consists
of all $f\in \mathscr{S}'(\mathbb{R}^n)$ such that the norm
\begin{equation}
\begin{split}
\|f\|_{W_{p,q}^{m,\omega}}&=\big\|\|V_{\phi}f(x,\xi)\|_{L_{\xi;q,m}}\big\|_{L_{x;p,\omega}}
\\&
=\left(\int_{\mathbb{R}^n}\left(\int_{\mathbb{R}^n}|V_{\phi}f(x,\xi)m(\xi)|^{q} d\xi\right)^{{p}/{q}}|\omega(x)|^pdx\right)^{{1}/{p}}
\end{split}
\end{equation}
is finite, with the usual modification when $p=\infty$ or $q=\infty$.
\end{definition}

Again, the definition is independent of the choice of the window $\phi $. We
write $W_{p,q}^{s,t}:=W_{p,q}^{m,\omega }$ for $\omega =\langle x\rangle
^{t} $ and $m(\xi )=\langle \xi \rangle ^{s}$.

Since $V_{\phi }f(x,\xi )=V_{\hat{\phi}}\hat{f}(\xi ,-x)$, we have the
following relations between the modulation space and the Wiener amalgam space:
\begin{equation}
W_{p,q}^{\omega,m}=\mathscr{F}(\mathcal {M}_{q,p}^{\tilde{\omega},m}),\hspace{10mm} \mathcal {M}_{p,q}^{m,\omega}=\mathscr{F}(W_{q,p}^{\tilde{\omega},m}),
\end{equation}
where $\tilde{\omega}(x)=\omega(-x)$.

Based on the above relations, properties of Wiener amalgam spaces may be
deduced directly from the corresponding properties of modulation spaces.
Thus, in this article we will mainly give the proof on the modulation space,
then the corresponding conclusion on the Wiener amalgam space follows.

As we known, some algebraic properties, such as the product and the
convolution, play a decisive role in the research of some nonlinear problems
in partial differential equation. However, we notice that extensive studies
on the modulation spaces emerged mostly in last ten years. Compared to the
classical Lebesgue spaces and Besov spaces, these properties, as well as
some analysis properties, of the modulation spaces are quite different and
still are not fully explored. Below we briefly review some historical
results on this subject.

One initial significant work is Feichtinger's paper \cite{Feichtinger_Banach
convolution}, which gives a general description of Banach convolution
property for Wiener type spaces $W(B,C)$ defined on locally compact groups.
One can also see \cite{C.Heil} for some convolution properties for weighted
Wiener amalgam spaces. For the unweighted modulation space $\mathcal{M}%
_{p,q} $, the embedding, product and convolution relations are characterized
by Cordero and Nicola in \cite{Cordero_Nicola_Sharpness}. Inspired by their
results, in an earlier paper \cite{Guo_SCI China}, we obtain the optimum of
product inequality on $\mathcal{M}_{p,q}^{0,t}$, convolution inequality on $%
\mathcal{M}_{p,q}^{s,0}$ and embedding on the weighted modulation spaces $%
\mathcal{M}_{p,q}^{s,t}$, for $1\leq p,q\leq \infty $, $s,t\in \mathbb{R}$.
Furthermore, we study the optimum of product inequality on $\mathcal{M}%
_{p,q}^{s}$ in a recent work \cite{Guo_Sharp convolution}. Meanwhile, we
notice that a recent paper \cite{Toft_sharp convolution} also concerns the
similar problem. In \cite{Toft_sharp convolution}, the authors establish
some sufficient conditions, as well as some necessary conditions, on the
product and convolution inequalities on weighted modulation spaces $\mathcal{%
M}_{p,q}^{s,t}$. However, their results remain a distance from the optimum,
since there is a gap between the sufficiency and necessity.

In this paper, we will continue this topic. As one of the targets,
we will give complete answers on \ $\mathcal{M}_{p,q}^{s,t}$ by establishing the
sharp product and convolution inequalities. However, our research will be,
not merely on $\mathcal{M}_{p,q}^{s,t}$, engaged in the more general space
$\mathcal{M}_{p,q}^{m,\omega }$. To this end, we start with asking a more
general question with weights in $\mathscr{P}(\mathbb{R}^{n})$: if $f$
lives in one modulation space $\mathcal{M}_{p_{1},q_{1}}^{m_{1},\omega _{1}}$
and $g$ lives in another modulation space $\mathcal{M}_{p_{2},q_{2}}^{m_{2},%
\omega _{2}}$, what modulation space $\mathcal{M}_{p,q}^{m,\omega }$ does
the product $fg$ or the convolution $f\ast g$ live in? More quantitatively,
what optimal conditions can guarantee the bilinear estimates
\begin{equation}\label{Introduction 1}
\|fg\|_{\mathcal {M}_{p,q}^{m,\omega}}\lesssim  \|f\|_{\mathcal {M}_{p_1,q_1}^{m_1,\omega_1}}\|g\|_{\mathcal {M}_{p_2,q_2}^{m_2,\omega_2}}
\end{equation}
or
\begin{equation}\label{Introduction 2}
\|f\ast g\|_{\mathcal {M}_{p,q}^{m,\omega}}\lesssim  \|f\|_{\mathcal {M}_{p_1,q_1}^{m_1,\omega_1}}\|g\|_{\mathcal {M}_{p_2,q_2}^{m_2,\omega_2}}.
\end{equation}
If $f\in \mathscr{S}^{\prime }$ and $g\in \mathscr{S}^{\prime }$, it is not
clear in general what is the sense of $fg$ and $f\ast g$. For this reason,
we use $\mathcal{M}_{p_{1},q_{1}}^{m_{1},\omega _{1}}\cdot \mathcal{M}%
_{p_{2},q_{2}}^{m_{2},\omega _{2}}\subset \mathcal{M}_{p,q}^{m,\omega }$ to
denote that the product map $f\cdot g$ initially defined on $\mathscr{S}%
\times \mathscr{S}$ extends to a bounded bilinear map from $\mathcal{M}%
_{p_{1},q_{1}}^{m_{1},\omega _{1}}\times \mathcal{M}_{p_{2},q_{2}}^{m_{2},%
\omega _{2}}$ into $\mathcal{M}_{p,q}^{m,\omega }$. Similarly, we use $%
\mathcal{M}_{p_{1},q_{1}}^{m_{1},\omega _{1}}\ast \mathcal{M}%
_{p_{2},q_{2}}^{m_{2},\omega _{2}}\subset \mathcal{M}_{p,q}^{m,\omega }$ to
denote that the convolution map $f\ast g$ initially defined on $\mathscr{S}%
\times \mathscr{S}$ extends to a bounded bilinear map from $\mathcal{M}%
_{p_{1},q_{1}}^{m_{1},\omega _{1}}\times \mathcal{M}_{p_{2},q_{2}}^{m_{2},%
\omega _{2}}$ into $\mathcal{M}_{p,q}^{m,\omega }$. More generally, for
function spaces $X,Y$ and $Z$ defined on $\mathbb{R}^{n}$, we adopt the
notations $X\cdot Y\subset Z$ and $X\ast Y\subset Z$ to denote the similar
meaning. We use the notation $X\subset Y$ to denote the continuous embedding
of function spaces.

Our strategy of research is to reduce the problems to their equivalent
discrete versions. So, next we introduce the discrete Lebesgue spaces. Let
$s\in \mathbb{R}$, $0<p,q\leq \infty $, $\omega :\mathbb{Z}^{n}\rightarrow \mathbb{R}^{+}$
be a weight function. The weighted discrete Lebesgue space $l_{k;p,\omega }$
consists of all functions $f:\mathbb{Z}^{n}\rightarrow \mathbb{C}$ whose norm
\begin{numcases}{\|f\|_{l_{k;p,\omega}}=}
\left(\sum_{k\in \mathbb{Z}^n}|f(k)\omega(k)|^p \right)^{{1}/{p}}, &$p<\infty$
\\
\sup_{k\in \mathbb{Z}^n}|f(k)\omega(k)|,\hspace{15mm} &$p=\infty$
\end{numcases}
is finite.
We write $l_{p,\omega}$ for short, if there is no chance of confusion.
We write $l_{p, s}$ for the case $\omega(k)=\langle k\rangle^s$.

If $\Omega$ is a compact subset of $\mathbb{R}^n$, we denote the function class
\begin{equation}
L^{\Omega}_{p,\omega}=\{f\in \mathscr{S}': \textbf{supp}\hat{f}\subset \Omega, \|f\|_{L_{p,\omega}}<\infty \}.
\end{equation}
Similarly, we write $L^{\Omega}_{p,t}$ for the case $\omega(x)=\langle x\rangle^t$.
We use $\mathscr{S}^{\Omega}$ to denote the set of all Schwartz functions with Fourier supports contained in $\Omega$.
We use $(G,~ d\mu)$ to denote $G=\mathbb{R}^n$ with the usual Lebesgue measure, or $G=\mathbb{Z}^n$ with the counting measure.
In many cases, we do not distinguish between functions defined on the $\mathbb{R}^n$ or $\mathbb{Z}^n$.
Likewise, we do not distinguish the product map, convolution map defined by functions on $\mathbb{R}^n$ or $\mathbb{Z}^n$.

For two functions defined on $G$, we define the product map
\begin{equation}
(f\cdot g)(x)=f(x)g(x)
\end{equation}
and the convolution map
\begin{equation}
(f\ast g)(x)=\int_{G}f(x-y)g(y)d\mu(y).
\end{equation}

We use $l_{q_1,s_1}\cdot l_{q_2,s_2}\subset l_{q,s}$ and $l_{q_1,s_1}\ast l_{q_2,s_2}\subset l_{q,s}$ to denote the relationship
\begin{equation*}
\|f\cdot g\|_{l_{q,s}}\lesssim \|f\|_{l_{q_1,s_1}}\|g\|_{l_{q_2,s_2}}
\end{equation*}
and
\begin{equation*}
\|f\ast g\|_{l_{q,s}}\lesssim \|f\|_{l_{q_1,s_1}}\|g\|_{l_{q_2,s_2}}
\end{equation*}
respectively, for all $f$ and $g$ defined on $\mathbb{Z}^n$.

Now, we are in a position to state our main results. We will state the
results in more general multi-linear versions, but for simplicity only give
the detailed proofs for the bilinear cases, since the proofs for the general
cases are essential the same. Also, for the sake of convenience, we use the
symbol $\coprod $ to represent multiple convolution, and the symbol $\prod $
to represent multiple product.


\begin{theorem}\label{characterization of product, multiple case}
Let $J\geq 2$ be an integer. Suppose $1\leq p\leq \infty$, $0<p_j, q, q_j\leq \infty$, $\omega,\omega_j, m,m_j\in \mathscr{P}({R^n})$ for $j=1,2,\cdots ,J$.
Let $\Gamma$ be a compact subset of $\mathbb{R}^n$ with non-empty interior.
Then
\begin{enumerate}
 \item
 \label{product inequality, modulation space, multiple case}$
  \prod_{j=1}^J \mathcal {M}_{p_j,q_j}^{m_j, \omega_j}\subset \mathcal {M}_{p,q}^{m, \omega}
  $
  if and only if one of the following statements (a) and
  (b) holds:
    \begin{enumerate}
     \item
     $\prod_{j=1}^J L^{\Gamma}_{p_j,\omega_j}\subset L_{p,\omega},~\coprod_{j=1}^J L^{\Gamma}_{q_j,m_j}\subset L_{q,m},$
     \item
     \label{discrete characterization, product inequality, modulation space, multiple case}
     $\prod_{j=1}^J l_{p_j,\omega_j}\subset l_{p,\omega},~\coprod_{j=1}^J l_{q_j,m_j}\subset l_{q,m}$.
   \end{enumerate}
 \item
  \label{convolution inequality, wiener space, multiple case}
  $\coprod_{j=1}^J W_{p_j,q_j}^{m_j, \omega_j}\subset W_{p,q}^{m, \omega}$
  if and only if one of the following statements (a) and
  (b) holds:
    \begin{enumerate}
     \item $\prod_{j=1}^J L^{\Gamma}_{q_j,m_j}\subset L_{q,m},~\coprod_{j=1}^J L^{\Gamma}_{p_j,\omega_j}\subset L_{p,\omega},$
     \item
     \label{discrete characterization, convolution inequality, wiener space, multiple case}
     $\prod_{j=1}^J l_{q_j,m_j}\subset l_{q,m},~\coprod_{j=1}^J l_{p_j,\omega_j}\subset l_{p,\omega}$.
   \end{enumerate}
\end{enumerate}
\end{theorem}

\begin{theorem}\label{characterization of convolution, multiple case}
Let $J\geq 2$ be an integer. Suppose $0< p,p_j,q,q_j\leq \infty$, $\omega,\omega_j, m,m_j\in \mathscr{P}({R^n})$ for $j=1,2,\cdots ,J$.
Let $\Gamma$ be a compact subset of $\mathbb{R}^n$ with non-empty interior.
Then
\begin{enumerate}
 \item\label{convolution inequality, modulation space, multiple case}
 $\coprod_{j=1}^J \mathcal {M}_{p_j,q_j}^{m_j, \omega_j}\subset \mathcal {M}_{p,q}^{m, \omega}$
  if and only if one of the following statements (a) and
  (b) holds:
    \begin{enumerate}
     \item
   $\prod_{j=1}^J L^{\Gamma}_{q_j,m_j}\subset L_{q,m},~\coprod_{j=1}^J L^{\Gamma}_{p_j,\omega_j}\subset L_{p,\omega},$
     \item
     \label{discrete characterization, convolution inequality, modulation space, multiple case}
     $\prod_{j=1}^J l_{q_j,m_j}\subset l_{q,m},~\coprod_{j=1}^J l_{p_j,\omega_j}\subset l_{p,\omega}$.
     \end{enumerate}
 \item \label{product inequality, wiener space, multiple case}
  $\prod_{j=1}^J W_{p_j,q_j}^{m_j, \omega_j}\subset W_{p,q}^{m, \omega}$
  if and only if one of the following statements (a) and
  (b) holds:
    \begin{enumerate}
    \item
    $\prod_{j=1}^J L^{\Gamma}_{p_j,\omega_j}\subset L_{p,\omega},~\coprod_{j=1}^J L^{\Gamma}_{q_j,m_j}\subset L_{q,m},$
     \item \label{discrete characterization, product inequality, wiener space, multiple case}
     $\prod_{j=1}^J l_{p_j,\omega_j}\subset l_{p,\omega},~\coprod_{j=1}^J l_{q_j,m_j}\subset l_{q,m}$.
     \end{enumerate}
\end{enumerate}
\end{theorem}

\begin{theorem}[Characterization of embedding]\label{characterization of embedding}
Suppose $0< p_1,p_2,q_1,q_2\leq \infty$, $\omega_j, m_j\in \mathscr{P}({R^n})$ for $j=1,2$.
Let $\Gamma$ be a compact subset of $\mathbb{R}^n$ with non-empty interior.
Then
\begin{enumerate}
 \item $\mathcal {M}_{p_1,q_1}^{\omega_1,m_1}\subset \mathcal {M}_{p_2,q_2}^{\omega_2,m_2}$ if and only if one of the following
  statements (a) and (b) holds:
     \begin{enumerate}
     \item  $L^{\Gamma}_{p_1,\omega_1}\subset L_{p_2,\omega_2},\,~L^{\Gamma}_{q_1,m_1}\subset L_{q_2,m_2},$
     \item  $l_{p_1,\omega_1}\subset l_{p_2,\omega_2},\,~l_{q_1,m_1}\subset l_{q_2,m_2}$.
     \end{enumerate}

 \item $W_{p_1,q_1}^{\omega_1,m_1}\subset W_{p_2,q_2}^{\omega_2,m_2}$ if and only if one of the following
  statements (a) and (b) holds:
    \begin{enumerate}
    \item $L^{\Gamma}_{p_1,\omega_1}\subset L_{p_2,\omega_2},\,~L^{\Gamma}_{q_1,m_1}\subset L_{q_2,m_2},$
    \item $l_{p_1,\omega_1}\subset l_{p_2,\omega_2},\,~l_{q_1,m_1}\subset l_{q_2,m_2}$.
    \end{enumerate}
\end{enumerate}
\end{theorem}

This paper is organized as follows. In Section 2, we first show that,
similar to the one weight case, the modulation space $\mathcal {M}_{p,q}^{m,\omega}$ (with two weights)
has an alternative definition which is discrete on the frequency plane. We will
define the discrete weighted modulation space \ $M_{p,q}^{m,\omega }$ \ and
prove that the norms $\left\Vert \cdot\right\Vert _{M_{p,q}^{m,\omega }}$ \ and
$\left\Vert \cdot\right\Vert _{\mathcal{M}_{p,q}^{m,\omega }}$ are equivalent in
Proposition 2.1. We use this discrete type of definition to give first
reduction of our main theorems. Section 3 is devoted to the second reduction
of our theorems, which is based on the discretization of band limited
functions (functions with compact Fourier support). To describe our
procedure more clearly, in Section 2 we will give the first reduction for
the proof of Theorem 1.3 by showing that \ $M_{p_{1},q_{1}}^{m_{1},\omega
_{1}}\cdot M_{p_{2},q_{2}}^{m_{2},\omega _{2}}\subset M_{p,q}^{m,\omega }$
if and only if $L_{p_{1},\omega _{1}}^{\Gamma }\cdot L_{p_{2},\omega
_{2}}^{\Gamma }\subset L_{p,\omega }$ and $l_{q_{1},m_{1}}\ast
l_{q_{2},m_{2}}\subset l_{q,m}$ (see Proposition 2.6) for some compact set \
$\Gamma \subset
\mathbb{R}
^{n}$. Then, in Section 3, we give the second reduction for the proof of
Theorem 1.3 by showing that $L_{p_{1},\omega _{1}}^{\Gamma }\cdot
L_{p_{2},\omega _{2}}^{\Gamma }\subset L_{p,\omega }$ $\ $if and only if $%
l_{p_{1},\omega _{1}}\cdot l_{p_{2},\omega _{2}}$ $\subset l_{p,\omega }$ \
(see Proposition 3.2). Clearly these two reductions complete the proof of
Theorem 1.3 in the bilinear case. Similarly, in Section 2 we will give the
first reduction for the proof of Theorem 1.4 by showing that \ $%
M_{p_{1},q_{1}}^{m_{1},\omega _{1}}\ast M_{p_{2},q_{2}}^{m_{2},\omega
_{2}}\subset M_{p,q}^{m,\omega }$ if and only if $L_{p_{1},\omega
_{1}}^{\Gamma }\ast L_{p_{2},\omega _{2}}^{\Gamma }\subset L_{p,\omega }$
and $l_{q_{1},m_{1}}\cdot l_{q_{2},m_{2}}\subset l_{q,m}$ (see Proposition
2.7) for some compact set \ $\Gamma \subset
\mathbb{R}
^{n}$. Then, in Section 3, we give the second reduction for the proof of
Theorem 1.4 by showing that $L_{p_{1},\omega _{1}}^{\Gamma }\ast
L_{p_{2},\omega _{2}}^{\Gamma }\subset L_{p,\omega }$ $\ $if and only if $%
l_{p_{1},\omega _{1}}\ast l_{p_{2},\omega _{2}}\subset l_{p,\omega }$ (see
Proposition 3.2). These two reductions complete the proof of Theorem 1.4 in
the bilinear case. For the imbedding, in Section 2 we will show that $%
M_{p_{1},q_{1}}^{m_{1},\omega _{1}}\subset M_{p_{2},q_{2}}^{m_{2},\omega
_{2}}$ if and only if $L_{p_{1},\omega _{1}}^{\Gamma }\subset
L_{p_{2},\omega _{2}}^{\Gamma }$ and $l_{q_{1},m_{1}}\subset
l_{q_{2},m_{2}}\ $(see Proposition 2.8) and in Section 3 we will show $%
L_{p_{1},\omega _{1}}^{\Gamma }\subset L_{p_{2},\omega _{2}}^{\Gamma }$ if
and only if $l_{p_{1},\omega _{1}}\subset l_{p_{2},\omega _{2}}$ (see
Proposition 3.1). These two reductions finish the proof for Theorem 1.5.
Additionally, at the end of Section 3, some remarks are
given for comparison with the known results. Finally, some applications are
presented in Section 4. We only focus on the cases which can be
characterized in a more concrete way. The index groups for the product
inequalities, convolution inequalities, embedding relations on modulation
spaces with power weights are completely characterized in this section.

Throughout this paper, we will adopt the following notations. Let $C$ be a
positive constant that may depend on $n,\,p_{i},\,q_{i},\,s_{i},\,t_{i},%
\omega _{i},m_{i},\,(i=1,\,2)$. The notation $X\lesssim Y$ denotes the
statement that $X\leq CY$, and the notation $X\lesssim
_{a_{1},a_{2},...,a_{k}}Y$ denotes the statement $X\leq
C_{a_{1},a_{2},...,a_{k}}Y$ for some positive constant $%
C_{a_{1},a_{2},...,a_{k}}$, which may depend on the parameters $%
a_{1},a_{2},...,a_{k}$. The notation $X\sim Y$ means the statement $%
X\lesssim Y\lesssim X$, and the notation $X\sim _{a_{1},a_{2},...,a_{k}}Y$
denotes the statement $X\lesssim _{a_{1},a_{2},...,a_{k}}Y\lesssim
_{a_{1},a_{2},...,a_{k}}X$ For a multi-index $k=(k_{1},k_{2},...,k_{n})\in
\mathbb{Z}^{n}$, we denote $|k|_{\infty }:=\max_{i=1,2,...,n}|k_{i}|$, and $%
\langle k\rangle :=(1+|k|^{2})^{{\ 1}/{2}}$.

\section{First reduction}
As we know, the frequency-uniform localization techniques can be used to
discretize the modulation space by giving the space an alternative
definition. For this fact, the reader can see \cite{Tribel_modulation
space,C.Heil} for some details for the discretizations of modulation space
and Wiener amalgam space, or see \cite{Feichtinger_Banach convolution} for
the discretization of Wiener type space in a more general frame. Using the
similar techniques, we will discretize the norm of weighted modulation space
$\mathcal {M}_{p,q}^{m,\omega }$, so that our proofs for the main theorems can be
executed on a discrete version.

For $k\in \mathbb{Z}^{n},$ we denote by $Q_{k}$ the unit closed cube
centered at $k$. We write $Q: =Q_0$ for short.
The family $\{Q_{k}\}_{k\in \mathbb{Z}^{n}}$ constitutes a
decomposition of $\mathbb{R}^{n}$. Let $\rho :\mathbb{R}^{n}\rightarrow \lbrack 0,1]$ be a smooth function
satisfying $\rho (\xi )=1$ for $|\xi |_{\infty }\leq {1}/{2}$ and $\rho
(\xi )=0$ for $|\xi |\geq 1$. Let $\rho _{k}$ be a translation of $\rho $,
\begin{equation*}
\rho _{k}(\xi )=\rho (\xi -k),\text{ }k\in \mathbb{Z}^{n}.
\end{equation*}
Since $\rho _{k}(\xi )=1$ for $\xi \in Q_{k}$, \ we have that $\sum_{k\in
\mathbb{Z}^{n}}\rho _{k}(\xi )\geq 1$ for all $\xi \in \mathbb{R}^{n}$.
Denote
\begin{equation*}
\sigma _{k}(\xi )=\rho _{k}(\xi )\left( \sum_{l\in \mathbb{Z}^{n}}\rho
_{l}(\xi )\right) ^{-1},~~~~k\in \mathbb{Z}^{n}.
\end{equation*}
The sequence $\{\sigma _{k}(\xi )\}_{k\in \mathbb{Z}^{n}}$ constitutes a smooth
decomposition of  $\ \mathbb{R}^{n},$ where $\sigma _{k}(\xi )=\sigma (\xi -k)$. The frequency-uniform decomposition operators are defined by
\begin{equation*}
\Box _{k}:=\mathscr{F}^{-1}\sigma _{k}\mathscr{F}
\end{equation*}
for $k\in \mathbb{Z}^{n}$. With the family \ $\left\{ \Box _{k}\right\}
_{k\in \mathbb{Z}^{n}}$,\ an alternative norm of modulation space can be
defined by
\begin{equation}
\|f\|_{M_{p,q}^{\omega,m}}=\left( \sum_{k\in \mathbb{Z}^{n}}\Vert \Box _{k}f\Vert_{L_{p,\omega}}^{q}|m(k)|^q\right) ^{{1}/{q}},
\end{equation}
with a natural modification for $q=\infty $, where $\omega$ and $m$ are weight functions defined on $\mathbb{R}^n$ and $\mathbb{Z}^n$ respectively.

\begin{proposition}\label{proposition, equivalent norm}
Let $m, \omega \in \mathscr{P}(\mathbb{R}^n)$, $0<p,q\leq \infty$.
Then the norm $\|\cdot\|_{M_{p,q}^{m,\omega}}$
is an equivalent quasi-norm on $\mathcal {M}_{p,q}^{m,\omega}$ with usual modification if $q=\infty$.
\end{proposition}

To verify the proposition, we first need the following weighted convolution inequality.

\begin{lemma}[\textbf{Weighted convolution in $L_{p}$ with $p<1$}]\label{lemma, convolution for p<1}
Suppose that $\omega\in \mathscr{P}(\mathbb{R}^n)$ is a $v$-moderate weight. Let $0<p<1$ , $\Omega$, $\Omega'$  be compact subsets of $\mathbb{R}^n$.
Suppose $f\in L^{\Omega}_{p,\omega}$, $g\in L_{p, v}^{\Omega'}$. Then
there exists a constant $C>0$ which depends only on the diameters of $\Omega$, $\Omega'$ and the exponent $p$,
such that
\begin{equation}\label{for proof, 2}
\Vert |f|\ast |g|\Vert _{L_{p, \omega}}\leq C\Vert f\Vert _{L_{p, \omega}}\Vert g\Vert _{L_{p, v}}.
\end{equation}
\end{lemma}
\begin{proof}
This lemma can be found in [16] when \ $\omega =1$ and \ $v=1$.
By a standard limiting argument, we only need to verify  (\ref{for proof, 2}) for $f\in \mathscr{S}^{\Omega}$ and $g\in \mathscr{S}^{\Omega'}$.
In this case, we observe that $H(y)=f(x-y)g(y)$ is a Schwartz function with Fourier support in $\Omega'-\Omega$ for every $x\in \mathbb{R}^n$.
We use the embedding relation $L_p^{\Omega'-\Omega}\subset L_1^{\Omega'-\Omega}$ to deduce
\begin{equation}
\int_{\mathbb{R}^n}|f(x-y)g(y)|dy\lesssim \left(\int_{\mathbb{R}^n}|f(x-y)g(y)|^pdy\right)^{1/p}.
\end{equation}
By the assumption $\omega(x)\lesssim \omega(x-y)v(y)$, it follows that
\begin{equation}
\left(\int_{\mathbb{R}^n}|f(x-y)g(y)|dy\right)\omega(x)\lesssim \left(\int_{\mathbb{R}^n}|f(x-y)\omega(x-y)g(y)v(y)|^pdy\right)^{1/p}.
\end{equation}
Taking the $L_{x;p}$-norm on both sides of the above inequality, we then use the Fubini theorem to obtain the desired conclusion.
\end{proof}

Now, we give the proof of Proposition \ref{proposition, equivalent norm}.
\\
\textbf{Proof of Proposition \ref{proposition, equivalent norm}.} We only give the proof for $q<\infty$, since the $q=\infty$ case can be handled similarly.
By the independence of window function in Definition \ref{Definition, modulation space, continuous form}, we can assume that
$\widehat{\phi}\subset B(0, 10\sqrt{n})$ and $\widehat{\phi}(\xi)=1$ in $B(0, 5\sqrt{n})$. For any $f\in \mathcal {S}'$,
\begin{equation}
\Box_kf=\mathscr{F}^{-1}\sigma_k\mathscr{F}f=\mathscr{F}^{-1}\sigma_kT_{\xi}\widehat{\phi}\mathscr{F}f
\end{equation}
for $\xi \in Q_k$. Using Young's inequality or Lemma \ref{lemma, convolution for p<1}, we deduce
\begin{equation}
\begin{split}
\|\Box_kf\|_{L_{p,\omega}}
\lesssim &
\|\mathscr{F}^{-1}\sigma_k\|_{L_{p\wedge 1},v}\|\mathscr{F}^{-1}T_{\xi}\widehat{\phi}\mathscr{F}f\|_{L_{p,\omega}}
\\
\lesssim &
\|\mathscr{F}^{-1}T_{\xi}\widehat{\phi}\mathscr{F}f\|_{L_{p,\omega}}
\end{split}
\end{equation}
for $\xi \in Q_k$. Observing that $m(\xi)\sim m(k)$ for $\xi\in Q_k$, we deduce
\begin{equation}
\begin{split}
\|\Box_kf\|_{L_{p,\omega}}|m(k)|
\lesssim &
\left(\int_{Q_k}\|\mathscr{F}^{-1}T_{\xi}\widehat{\phi}\mathscr{F}f\|^q_{L_{p,\omega}}|m(\xi)|^qd\xi\right)^{1/q}
\\
= &
\left(\int_{Q_k}\|V_{\phi}f(x,\xi)\|^q_{L_{p,\omega}}|m(\xi)|^qd\xi\right)^{1/q}.
\end{split}
\end{equation}
Taking $l_q$-norm on both sides of the above inequality, we obtain
\begin{equation}
\|f\|_{M_{p,q}^{\omega,m}}\lesssim \|f\|_{\mathcal {M}_{p,q}^{\omega,m}}.
\end{equation}
On the other hand, for $\xi \in Q_k$,
\begin{equation}
V_{\phi}f(x, \xi)=(\mathscr{F}^{-1}T_{\xi}\widehat{\phi}\mathscr{F}f)(x)=\sum_{|l-k|\leq c}(\mathscr{F}^{-1}\sigma_lT_{\xi}\widehat{\phi}\mathscr{F}f)(x).
\end{equation}
Thus
\begin{equation}
\begin{split}
\|V_{\phi}f(x,\xi)\|_{L_{p,\omega}}
\lesssim &
\|\sum_{|l-k|\leq c}(\mathscr{F}^{-1}\sigma_lT_{\xi}\widehat{\phi}\mathscr{F}f)(x)\|_{L_{p,\omega}}
\\
\lesssim &
\sum_{|l-k|\leq c}\|\mathscr{F}^{-1}[\widehat{\phi}(\cdot -\xi)]\|_{L_{p\wedge 1,v}}\|\Box_l f\|_{L_{p, \omega}}
\\
\lesssim &
\sum_{|l-k|\leq c}\|\Box_l f\|_{L_{p, \omega}}.
\end{split}
\end{equation}
for $\xi \in Q_k$.
Taking the $q$-power integration over $\xi\in Q_k$ with weight $m$, we deduce
\begin{equation}
\begin{split}
\int_{Q_k}\|V_{\phi}(x,\xi)\|^q_{L_{p,\omega}}|m(\xi)|^qd\xi
\lesssim
\sum_{|l-k|\leq c}\|\Box_l f\|^q_{L_{p, \omega}}|m(l)|^q.
\end{split}
\end{equation}
Summation over $k$ leads to
\begin{equation}
\|f\|_{\mathcal {M}_{p,q}^{m,\omega}}\lesssim \|f\|_{M_{p,q}^{m,\omega}}.
\end{equation}
Thanks to the above Proposition \ref{proposition, equivalent norm}, we will use the discrete form norm $\|\cdot\|_{M_{p,q}^{\omega,m}}$
instead of the continuous form norm $\|\cdot\|_{\mathcal {M}_{p,q}^{m,\omega}}$
throughout the rest of our paper.

Now, we begin the process of discretizing our main theorems. As mentioned
before, the conclusion about Wiener amalgam space can be deduced by the
corresponding conclusion of modulation space, and the multi-linear case can
be induced by the bilinear case. So we only give the detailed proof for the
bilinear case associated with modulation space. Firstly, we point out that
for a tempered distribution $f\in \mathscr{S}^{\prime }$ with compact
Fourier support $K$, the modulation space norm $\Vert f \Vert
_{M_{p,q}^{m,\omega }}$ is equivalent to the weighted $L^{p}$ norm $\Vert
f \Vert _{L_{p,\omega }}$.

\begin{lemma}\label{lemma, equivalent norm, compact support}
Let $K$ be a compact subset of $\mathbb{R}^n$, and $f$ be a tempered distribution with Fourier support contained in $K$.
Then $f\in M_{p,q}^{m, \omega}$ if and only if $f\in L^{p, \omega}$, and
\begin{equation}
\|f\|_{M_{p,q}^{m, \omega}}\sim_K  \|f\|_{L_{p,\omega}}.
\end{equation}
\end{lemma}

The proof is based on the finite covering on $K$. We leave its detail to the
reader. However, one can see \cite[Lemma 3.2]{Cordero_Nicola_Sharpness} for
the proof in unweighted case.

We also need the following technical lemma.

\begin{lemma}\label{lemma, independence of Gamma, product}
Suppose $0<p, p_j\leq \infty$, $\omega,\omega_j \in \mathscr{P}({R^n})$ for $j=1,2$.
Let $\Omega$, $\Gamma$ be compact subsets of $\mathbb{R}^n$ with non-empty interior.
Then
\begin{equation}
L^{\Omega}_{p_1, \omega_1}\cdot L^{\Omega}_{p_2, \omega_2} \subset L_{p, \omega}
\end{equation}
if and only if
\begin{equation}
L^{\Gamma}_{p_1, \omega_1}\cdot L^{\Gamma}_{p_2, \omega_2} \subset L_{p, \omega}.
\end{equation}
\end{lemma}
\begin{proof}
By the symmetry of $\Gamma$ and $\Omega$,
we only need to give the proof for sufficiency.
By the assumption, there exists an $x_0\in \mathbb{R}^n$, such that $B(x_0, r)\subset \Gamma$ for some $r>0$.
For the compactness of $\Omega$, we can find a smooth function $h(\xi)$ with compact support near the origin,
and find a point sequence  $\kappa_j\in \mathbb{R}^n$, $j\in A$,
where $A$ is a finite set,
such that
\begin{equation}
\begin{cases}
\textbf{supp}h_j\subset B(\kappa_j,r),
\\
\sum_{j\in A}h_j(\xi)\equiv 1~in~\Omega,
\end{cases}
\end{equation}
where $h_j(\xi)=h(\xi-\kappa_j)$.
Write $P_j=\mathscr{F}^{-1}h_j \mathscr{F}$.
\ For two functions $f\in L^{\Omega}_{p_1, \omega_1}\cap \mathscr{S}$, $g\in L^{\Omega}_{p_2, \omega_2}\cap \mathscr{S}$,
we have
\begin{equation}
M_{x_0-\kappa_i}P_i(f)\in L^{\Gamma}_{p_1, \omega_1}\cap \mathscr{S} ,\  M_{x_0-\kappa_j}P_j(g)\in L^{\Gamma}_{p_2, \omega_2}\cap \mathscr{S}.
\end{equation}
By the assumption $L^{\Gamma}_{p_1, \omega_1}\cdot L^{\Gamma}_{p_2, \omega_2} \subset L_{p, \omega}$, we obtain that
\begin{equation}
\begin{split}
\|P_i(f)\cdot P_j(g)\|_{L_{p, \omega}}
= &
\|M_{x_0-\kappa_i}P_i(f)\cdot M_{x_0-\kappa_j}P_j(g)\|_{L_{p, \omega}}.
\\
\lesssim &
\|M_{x_0-\kappa_i}P_i(f)\|_{L_{p_1, \omega_1}}\|M_{x_0-\kappa_j}P_j(g)\|_{L_{p_2, \omega_2}}
\\
= &
\|P_i(f)\|_{L_{p_1, \omega_1}}\|P_j(g)\|_{L_{p_2, \omega_2}}.
\end{split}
\end{equation}

Then, we use Young's inequality or Lemma \ref{lemma, convolution for p<1} to deduce
\begin{equation}
\begin{split}
\|P_i(f)\|_{L_{p_1, \omega_1}}
\lesssim &
\|\mathscr{F}^{-1}h_j\|_{L_{p_1\wedge 1,v}}\|f\|_{L_{p_1,\omega_1}}
\\
= &
\|\mathscr{F}^{-1}h\|_{L_{p_1\wedge 1,v}}\|f\|_{L_{p_1,\omega_1}}
\\
\lesssim &
\|f\|_{L_{p_1,\omega_1}}.
\end{split}
\end{equation}
Similarly, we deduce $\|P_j(g)\|_{L_{p_2, \omega_2}}\lesssim \|g\|_{L_{p_2, \omega_2}}$. Thus, we obtain that

\begin{equation}
\begin{split}
\|P_i(f)\cdot P_j(g)\|_{L_{p, \omega}}
\lesssim &
\|P_i(f)\|_{L_{p_1, \omega_1}}\|P_j(g)\|_{L_{p_2, \omega_2}}
\\
\lesssim &
\|f\|_{L_{p_1, \omega_1}}\|g\|_{L_{p_2, \omega_2}}.
\end{split}
\end{equation}

Recalling $|A|\leq \infty$, we obtain that
\begin{equation}
\begin{split}
\|fg\|_{L_{p, \omega}}
= &
\|\sum_{i,j\in A}P_i(f)\cdot P_j(g)\|_{L_{p, \omega}}
\\
\lesssim &
\sum_{i,j\in A}\|P_i(f)\cdot P_j(g)\|_{L_{p, \omega}}
\\
\lesssim &
|A|^2\|f\|_{L_{p_1, \omega_1}}\|g\|_{L_{p_2, \omega_2}}\lesssim \|f\|_{L_{p_1, \omega_1}}\|g\|_{L_{p_2, \omega_2}}.
\end{split}
\end{equation}
for any two functions $f\in L^{\Omega}_{p_1, \omega_1}\cap \mathscr{S}$, $g\in L^{\Omega}_{p_2, \omega_2}\cap \mathscr{S}$.
\end{proof}

Similarly, we can verify the following lemma, whose proof is similar to the above one. We omit the proof here.

\begin{lemma}\label{lemma, independence of Gamma, convolution}
Suppose $0<p, p_j\leq \infty$, $\omega,\omega_j \in \mathscr{P}({R^n})$ for $j=1,2$.
Let $\Omega$, $\Gamma$ be compact subsets of $\mathbb{R}^n$ with non-empty interior, $j=1,2$.
Then
\begin{equation}
L^{\Omega}_{p_1, \omega_1}\ast L^{\Omega}_{p_2, \omega_2} \subset L_{p, \omega}
\end{equation}
if and only if
\begin{equation}
L^{\Gamma}_{p_1, \omega_1}\ast L^{\Gamma}_{p_2, \omega_2} \subset L_{p, \omega}.
\end{equation}
\end{lemma}

We now give the following propositions for the first reduction of our main theorems.

\begin{proposition}[\textbf{First reduction, product}]\label{proposition, first reduction, product}
Suppose $1\leq p\leq \infty$, $0<p_j, q, q_j\leq \infty$, $\omega,\omega_j, m,m_j\in \mathscr{P}({R^n})$ for $j=1,2$.
Let $\Gamma$ be a compact subset of $\mathbb{R}^n$ whose interior  $int(\Gamma)$ is not empty.
Then
\begin{equation}
M_{p_1,q_1}^{m_1, \omega_1}\cdot M_{p_2,q_2}^{m_2, \omega_2} \subset M_{p, q}^{m, \omega}
\end{equation}
if and only if
\begin{equation}
L^{\Gamma}_{p_1, \omega_1}\cdot L^{\Gamma}_{p_2, \omega_2} \subset L_{p, \omega}
\end{equation}
and
\begin{equation}
l_{q_1, m_1}\ast l_{q_2, m_2} \subset l_{q, m}.
\end{equation}
\end{proposition}
\begin{proof}
We first show the necessity part.
By Lemma \ref{lemma, equivalent norm, compact support}, we conclude that
\begin{equation}
\begin{split}
\|fg\|_{L_{p,\omega}}\sim & \|fg\|_{M_{p, q}^{m, \omega}}
\\
\lesssim &
\|f\|_{M_{p_1, q_1}^{m_1, \omega_1}}\|g\|_{M_{p_2, q_2}^{m_2, \omega_2}}
\\
\sim &
\|f\|_{L_{p_1, \omega_1}}\|g\|_{L_{p_2, \omega_2}}
\end{split}
\end{equation}
for any two functions $f\in L^{\Gamma}_{p_1, \omega_1}\cap \mathscr{S}$, $g\in L^{\Gamma}_{p_2, \omega_2}\cap \mathscr{S}$,
which implies $L^{\Gamma}_{p_1, \omega_1}\cdot L^{\Gamma}_{p_2, \omega_2} \subset L_{p, \omega}$.

Then, we choose a nonzero smooth function $h$ with sufficiently small Fourier support near the origin.
Let $\vec{a}=\{a_k\}_{k\in \mathbb{Z}^n}$ and $\vec{b}=\{b_k\}_{k\in \mathbb{Z}^n}$ be two nonnegative functions (sequences) defined on $\mathbb{Z}^n$.
By assuming that the following two series converge, we define two functions
\begin{equation}
f(x)=\sum_{k\in \mathbb{Z}^n}a_kh(x)e^{2\pi i k\cdot x},\hspace{6mm}g(x)=\sum_{k\in \mathbb{Z}^n}b_kh(x)e^{2\pi i k\cdot x}.
\end{equation}
It follows
\begin{equation}
(fg)(x)=\sum_{j,l\in \mathbb{Z}^n}a_jb_lh^2(x) e^{2\pi i (j+l)\cdot x}.
\end{equation}
Observing that
\begin{eqnarray}
\Vert \Box _{k}(fg)\Vert _{L_{p, \omega}} &=&\Vert
\sum_{j+l=k}a_{j}b_{l}h^{2}(x)e^{2\pi i(j+l)\cdot x}\Vert _{L_{p, \omega}}\sim
\sum_{j+l=k}a_{j}b_{l}, \\
\Vert \Box _{k}(f)\Vert _{L_{p_1, \omega_1}} &\sim &a_{k},\text{ \ }\Vert \Box
_{k}(g)\Vert _{L_{p_2, \omega_2}}\sim b_{k}\text{ \ for all \ }k\in \mathbb{Z}^{n},
\end{eqnarray}
we use the definition of modulation space (the discrete form) to deduce
\begin{equation}
\|f\|_{M_{p_1,q_1}^{m_1,\omega_1}}\sim \|\vec{a}\|_{l_{q_1,m_1}},
\|g\|_{M_{p_2,q_2}^{m_2,\omega_2}}\sim \|\vec{b}\|_{l_{q_2,m_2}}
\end{equation}
and
\begin{equation}
\|fg\|_{M_{p,q}^{m,\omega}}\sim \|\vec{a}\ast \vec{b}\|_{l_{q,m}}.
\end{equation}

By the assumption $M_{p_1,q_1}^{m_1, \omega_1}\cdot M_{p_2,q_2}^{m_2, \omega_2} \subset M_{p, q}^{m, \omega}$, we obtain that
\begin{equation}
\|\vec{a}\ast \vec{b}\|_{l_{q,m}}\lesssim \|\vec{a}\|_{l_{q_1,m_1}}\|\vec{b}\|_{l_{q_2,m_2}},
\end{equation}
which implies $l_{q_1, m_1}\ast l_{q_2, m_2} \subset l_{q, m}$.

We next turn to show the sufficiency of this proposition.
Using the almost
orthogonality of the frequency projections $\sigma _{k}$, we have that for all  \ $k\in \mathbb{Z}^{n}$,
\begin{equation}
\Box_k(fg)=\sum_{i,j\in \mathbb{Z}^n}\Box_k(\Box_if\cdot \Box_jg)
=\sum_{|l|\leq c(n)}\sum_{i+j=k+l}\Box_k(\Box_if\cdot \Box_jg),
\end{equation}
where \ $c(n)$ \ is a constant depending only on \ $n$.
By the fact that $\Box_k$ is an uniform $L_{p,\omega}$ multiplier (using Young's inequality or Lemma \ref{lemma, convolution for p<1}), we obtain
\begin{equation}
\begin{split}
\|\Box_k(fg)\|_{L_{p,\omega}}=&\|\sum_{|l|\leq c(n)}\sum_{i+j=k+l}\Box_k(\Box_if\cdot \Box_jg)\|_{L_{p,\omega}}
\\
\lesssim &
\sum_{|l|\leq c(n)}\sum_{i+j=k+l}\|\Box_if \cdot \Box_jg\|_{L_{p,\omega}}.
\end{split}
\end{equation}
By the assumption $L^{\Gamma}_{p_1, \omega_1}\cdot L^{\Gamma}_{p_2, \omega_2} \subset L_{p, \omega}$
and Lemma \ref{lemma, independence of Gamma, product}, we further obtain that
\begin{equation}
\begin{split}
\|\Box_if \cdot \Box_jg\|_{L_{p,\omega}}
= &
\|M_{-i}\Box_if \cdot M_{-j}\Box_jg\|_{L_{p,\omega}}
\\
\lesssim &
\|M_{-i}\Box_if\|_{L_{p_1,\omega_1}} \cdot \|M_{-j}\Box_jg\|_{L_{p_2,\omega_2}}
\\
= &
\|\Box_if\|_{L_{p_1,\omega_1}} \cdot \|\Box_jg\|_{L_{p_2,\omega_2}}.
\end{split}
\end{equation}
Hence
\begin{equation}
\begin{split}
\|\Box_k(fg)\|_{L_{p,\omega}}\lesssim &\sum_{|l|\leq c(n)}\sum_{i+j=k+l}\|\Box_if \cdot \Box_jg\|_{L_{p,\omega}}
\\
\lesssim &
\sum_{|l|\leq c(n)}\sum_{i+j=k+l}\|\Box_if\|_{L_{p_1,\omega_1}} \cdot \|\Box_jg\|_{L_{p_2,\omega_2}}
\\
\lesssim &
\sum_{|l|\leq c(n)}\left(\{\|\Box_if\|_{L_{p_1,\omega_1}}\}\ast \{\|\Box_jg\|_{L_{p_2,\omega}}\}\right)(k+l).
\end{split}
\end{equation}

By the assumption $l_{q_1, m_1}\ast l_{q_2, m_2} \subset l_{q, m}$ and the fact $m(k)\sim m(k+l)$ for $|l|\leq c(n)$, we now conclude that
\begin{equation}
\begin{split}
\|fg\|_{M_{p,q}^{m,\omega}}
= &
\big\|\{\|\Box_k(fg)\|_{L_{p,\omega}}\}\big\|_{l_{k;q,m}}
\\
\lesssim &
\big\|\big\{\sum_{|l|\leq c(n)}\left(\{\|\Box_if\|_{L_{p_1,\omega_1}}\}\ast \{\|\Box_jg\|_{L_{p_2,\omega}}\}\right)(k+l)\big\}_{k\in \mathbb{Z}^n}\big\|_{l_{k;q,m}}
\\
\lesssim &
\big\|\{\left(\{\|\Box_if\|_{L_{p_1,\omega_1}}\}\ast \{\|\Box_jg\|_{L_{p_2,\omega}}\}\right)(k)\big\}_{k\in \mathbb{Z}^n}\big\|_{l_{k;q,m}}
\\
\lesssim &
\big\|\{\|\Box_kf\|_{L_{p_1,\omega_1}}\}\big\|_{l_{k;q_1,m_1}}
\cdot
\big\|\{\|\Box_kg\|_{L_{p_2,\omega_2}}\}\big\|_{l_{k;q_2,m_2}}
\\
=&
\|f\|_{M_{p_1,q_1}^{m_1,\omega_1}}\|g\|_{M_{p_2,q_2}^{m_2,\omega_2}}.
\end{split}
\end{equation}

\end{proof}

\begin{proposition}[\textbf{First reduction, convolution}]\label{proposition, first reduction, convolution}
Suppose $0<p, p_j, q, q_j\leq \infty$, $\omega,\omega_j, m,m_j\in \mathscr{P}({R^n})$ for $j=1,2$.
Let $\Gamma$ be a compact subset of $\mathbb{R}^n$ whose interior  $int(\Gamma)$ is not empty.
Then
\begin{equation}
M_{p_1,q_1}^{m_1, \omega_1}\ast M_{p_2,q_2}^{m_2, \omega_2} \subset M_{p, q}^{m, \omega}
\end{equation}
if and only if
\begin{equation}
L^{\Gamma}_{p_1, \omega_1}\ast L^{\Gamma}_{p_2, \omega_2} \subset L_{p, \omega}
\end{equation}
and
\begin{equation}
l_{q_1, m_1}\cdot l_{q_2, m_2} \subset l_{q, m}.
\end{equation}
\end{proposition}
\begin{proof}
We first show the necessity part.
By Lemma \ref{lemma, equivalent norm, compact support}, we conclude that
\begin{equation}
\begin{split}
\|f\ast g\|_{L_{p,\omega}}\sim & \|f\ast g\|_{M_{p, q}^{m, \omega}}
\\
\lesssim &
\|f\|_{M_{p_1, q_1}^{m_1, \omega_1}}\|g\|_{M_{p_2, q_2}^{m_2, \omega_2}}
\\
\sim &
\|f\|_{L_{p_1, \omega_1}}\|g\|_{L_{p_2, \omega_2}}
\end{split}
\end{equation}
for any two functions $f\in L^{\Gamma}_{p_1, \omega_1}\cap \mathscr{S}$, $g\in L^{\Gamma}_{p_2, \omega_2}\cap \mathscr{S}$.

Then, we choose a smooth function $h$ with sufficiently small Fourier support near the origin.
Let $\vec{a}=\{a_k\}_{k\in \mathbb{Z}^n}, \vec{b}=\{b_k\}_{k\in \mathbb{Z}^n}$ be two nonnegative functions (sequences) defined on $\mathbb{Z}^n$.
By assuming that the following two series converge, we define two functions
\begin{equation}
\widehat{f}(\xi)=\sum_{k\in \mathbb{Z}^n}a_k\widehat{h}(\xi-k),\hspace{6mm}\widehat{g}(\xi)=\sum_{k\in \mathbb{Z}^n}b_k\widehat{h}(\xi-k).
\end{equation}
Then
\begin{equation}
(\widehat{f}\cdot \widehat{g})(\xi)=\sum_{k\in \mathbb{Z}^n}a_kb_k(\widehat{h}\cdot \widehat{h})(\xi-k) .
\end{equation}
Observing that
\begin{eqnarray}
\Vert \Box _{k}(f\ast g)\Vert _{L_{p, \omega}} &=&\Vert
a_kb_k\mathscr{F}^{-1}[(\widehat{h}\cdot \widehat{h})(\cdot-k)]\Vert _{L_{p, \omega}}\sim
a_kb_k, \\
\Vert \Box _{k}(f)\Vert _{L_{p_1, \omega_1}} &\sim &a_{k},\text{ \ }\Vert \Box
_{k}(g)\Vert _{L_{p_2, \omega_2}}\sim b_{k}\text{ \ for all \ }k\in \mathbb{Z}^{n},
\end{eqnarray}
we use the definition of modulation space (the discrete form) to deduce
\begin{equation}
\|f\|_{M_{p_1,q_1}^{m_1,\omega_1}}\sim \|\vec{a}\|_{l_{q_1,m_1}},
\|g\|_{M_{p_2,q_2}^{m_2,\omega_2}}\sim \|\vec{b}\|_{l_{q_2,m_2}}
\end{equation}
and
\begin{equation}
\|f\ast g\|_{M_{p,q}^{m,\omega}}\sim \|\vec{a}\cdot \vec{b}\|_{l_{q,m}}.
\end{equation}

By the assumption $M_{p_1,q_1}^{m_1, \omega_1}\cdot M_{p_2,q_2}^{m_2, \omega_2} \subset M_{p, q}^{m, \omega}$, we obtain that
\begin{equation}
\|\vec{a}\cdot \vec{b}\|_{l(q,m)}\lesssim \|\vec{a}\|_{l(q_1,m_1)}\|\vec{b}\|_{l(q_2,m_2)},
\end{equation}
which implies $l_{q_1, m_1}\cdot l_{q_2, m_2} \subset l_{q, m}$.

We turn to show the sufficiency of this proposition.
Using the almost
orthogonality of the frequency projections $\{\sigma _{k}\}$, we have that for all  \ $k\in \mathbb{Z}^{n}$,
\begin{equation}
\Box_k(f\ast g)=\sum_{i,j\in \mathbb{Z}^n}\Box_k(\Box_if\ast \Box_jg)
=\sum_{|l|\leq c(n)}\sum_{|\tilde{l}|\leq c(n)}\Box_k(\Box_{k+l}f\ast \Box_{k+\tilde{l}}g),
\end{equation}
where \ $c(n)$ \ is a constant depending only on \ $n$.
By the fact that $\Box_k$ is an uniform $L_{p,\omega}$ multiplier (using Young's inequality or Lemma \ref{lemma, convolution for p<1}), we obtain that
\begin{equation}
\begin{split}
\|\Box_k(f\ast g)\|_{L_{p,\omega}}=&\|\sum_{|l|\leq c(n)}\sum_{|\tilde{l}|\leq c(n)}\Box_k(\Box_{k+l}f\ast \Box_{k+\tilde{l}}g)\|_{L_{p,\omega}}
\\
\lesssim &
\sum_{|l|\leq c(n)}\sum_{|\tilde{l}|\leq c(n)}\|\Box_{k+l}f\ast \Box_{k+\tilde{l}}g\|_{L_{p,\omega}}.
\end{split}
\end{equation}
By the assumption $L^{\Gamma}_{p_1, \omega_1}\ast L^{\Gamma}_{p_2, \omega_2} \subset L_{p, \omega}$
and Lemma \ref{lemma, independence of Gamma, convolution}, we obtain that
\begin{equation}
\begin{split}
\|\Box_{k+l}f\ast \Box_{k+\tilde{l}}g\|_{L_{p,\omega}}
= &
\|M_{-k}(\Box_{k+l}f \ast \Box_{k+\tilde{l}}g)\|_{L_{p,\omega}}
\\
= &
\|(M_{-k}\Box_{k+l}f) \ast (M_{-k}\Box_{k+\tilde{l}}g)\|_{L_{p,\omega}}
\\
\lesssim &
\|M_{-k}\Box_{k+l}f\|_{L_{p_1,\omega_1}} \cdot \|M_{-k}\Box_{k+\tilde{l}}g\|_{L_{p_2,\omega_2}}
\\
= &
\|\Box_{k+l}f\|_{L_{p_1,\omega_1}} \cdot \|\Box_{k+\tilde{l}}g\|_{L_{p_2,\omega_2}}
\end{split}
\end{equation}
for $|l|\leq c(n)$ and $|\tilde{l}|\leq c(n)$.
Hence
\begin{equation}
\begin{split}
\|\Box_k(f\ast g)\|_{L_{p,\omega}}\lesssim &\sum_{|l|\leq c(n)}\sum_{|\tilde{l}|\leq c(n)}\|\Box_{k+l}f \ast \Box_{k+\tilde{l}}g\|_{L_{p,\omega}}
\\
\lesssim &
\sum_{|l|\leq c(n)}\sum_{|\tilde{l}|\leq c(n)}\|\Box_{k+l}f\|_{L_{p_1,\omega_1}} \cdot \|\Box_{k+\tilde{l}}g\|_{L_{p_2,\omega_2}}.
\end{split}
\end{equation}

By the assumption $l_{q_1, m_1}\cdot l_{q_2, m_2} \subset l_{q, m}$ and the fact $m(k)\sim m(k+l)$ for $|l|\leq c(n)$, we conclude that
\begin{equation}
\begin{split}
\|f\ast g\|_{M_{p,q}^{m,\omega}}
= &
\big\|\{\|\Box_k(f\ast g)\|_{L_{p,\omega}}\}\big\|_{l_{k;q,m}}
\\
\lesssim &
\big\|\{\sum_{|l|\leq c(n)}\sum_{|\tilde{l}|\leq c(n)}\|\Box_{k+l}f\|_{L_{p_1,\omega_1}} \cdot \|\Box_{k+\tilde{l}}g\|_{L_{p_2,\omega_2}}\}_{k\in \mathbb{Z}^n}\big\|_{l_{k;q,m}}
\\
\lesssim &
\big\|\{\|\Box_kf\|_{L_{p_1,\omega_1}}\}\big\|_{l_{k;q_1,m_1}}
\cdot
\big\|\{\|\Box_kg\|_{L_{p_2,\omega_2}}\}\big\|_{l_{k;q_2,m_2}}
\\
=&
\|f\|_{M_{p_1,q_1}^{m_1,\omega_1}}\|g\|_{M_{p_2,q_2}^{m_2,\omega_2}}.
\end{split}
\end{equation}
\end{proof}

\begin{proposition}[\textbf{First reduction, embedding}]\label{proposition, first reduction, embedding}
Suppose $0<p_j, q_j\leq \infty$, $\omega_j, m_j\in \mathscr{P}({R^n})$ for $j=1,2$.
Let $\Gamma$ be a compact subset of $\mathbb{R}^n$ whose interior  $int(\Gamma)$ is not empty.
Then
\begin{equation}
M_{p_1,q_1}^{m_1, \omega_1}\subset  M_{p_2,q_2}^{m_2, \omega_2}
\end{equation}
if and only if
\begin{equation}
L^{\Gamma}_{p_1, \omega_1}\subset L^{\Gamma}_{p_2, \omega_2}
\end{equation}
and
\begin{equation}
l_{q_1, m_1}\subset l_{q_2, m_2}.
\end{equation}
\end{proposition}
Since this proposition can be viewed as
a degenerate form of Proposition \ref{proposition, first reduction, product} or Proposition \ref{proposition, first reduction, convolution},
the proof follows closely the proof of Proposition \ref{proposition, first reduction, product} and Proposition \ref{proposition, first reduction, convolution},
we omit its proof. Below, we have a degenerate version of Lemma \ref{lemma, independence of Gamma, product} and Lemma \ref{lemma, independence of Gamma, convolution}.
\begin{lemma}\label{lemma, independence of Gamma, embedding}
Suppose $0<p_j\leq \infty$, $\omega_j \in \mathscr{P}({R^n})$ for $j=1,2$.
Let $\Omega$, $\Gamma$ be compact subsets of $\mathbb{R}^n$ with non-empty interior, $j=1,2$.
Then
\begin{equation}
L^{\Omega}_{p_1, \omega_1}\subset L_{p_2, \omega_2}
\end{equation}
if and only if
\begin{equation}
L^{\Gamma}_{p_1, \omega_1}\subset L_{p_2, \omega_2}.
\end{equation}
\end{lemma}

\section{Second reduction}
In this section, we continue the process of discretization.
Our methods are based on the fact that a band limited function (function with compact Fourier support) can be reconstructed by its values at sufficiently dense discrete points,
which is the core of Shannon's sampling theorem. One can find some related discussions in Chapter 1 of Tribel's book \cite{Tribel_book_1983}.

\begin{proposition}[\textbf{Second reduction, embedding}]\label{proposition, second reduction, embedding}
Suppose $0<p_j\leq \infty$, $\omega_j\in \mathscr{P}({R^n})$ for $j=1,2$.
Let $\Gamma$ be a compact subset of $\mathbb{R}^n$ whose interior  $int(\Gamma)$ is not empty.
Then
\begin{equation}
L^{\Gamma}_{p_1, \omega_1}\subset L_{p_2, \omega_2}
\end{equation}
if and only if
\begin{equation}
l_{p_1, \omega_1}\subset l_{p_2, \omega_2}.
\end{equation}
\end{proposition}
\begin{proof}
We only give the proof for the case $p_j<\infty\  (j=1,2)$, since the other cases can be handled similarly.
\\
\textbf{Necessity.}
Let $\varphi$ be a nonnegative smooth function with compact Fourier support contained in $\frac{1}{2}Q$, such that $\varphi(0)>0$.
We choose a nonnegative sequence $\vec{a}=\{a_k\}_{k\in \mathbb{Z}^n}\in l(p_1, \omega_1)$, and write
\begin{equation}\label{for proof, 1}
C_{\vec{a}, \varphi}(x)=\sum_{k\in \mathbb{Z}^n}a_k\varphi(x-k).
\end{equation}
By the assumption that $\vec{a}\in l(p_1, \omega_1)\subset l(\infty, \omega_1)$, $\omega_1\in \mathscr{P}(\mathbb{R}^n)$,
we can verify that the right hand side of (\ref{for proof, 1}) at least converges in the sense of $\mathscr{S}'$.
Thus, $C_{\vec{a}, \varphi}(x)\in \mathscr{S}'$ and its Fourier support is contained in $\frac{1}{2}Q$.

By the rapidly decay of $\varphi$, for $x\in Q_l$, we obtain that
\begin{equation}
|C_{\vec{a}, \varphi}(x)|\lesssim \sum_{k\in \mathbb{Z}^n}a_k\varphi(x-k)\lesssim \sum_{k\in \mathbb{Z}^n}a_k\langle l-k\rangle^{-N}
\end{equation}
for any fixed positive constant $N$. Thus,
\begin{equation}
\begin{split}
\left(\int_{Q_l}|C_{\vec{a}, \varphi}(x)|^{p_1}|\omega_1(x)|^{p_1}dx\right)^{1/{p_1}}
\lesssim &
\sum_{k\in \mathbb{Z}^n}a_k\langle l-k\rangle^{-N}\omega_1(l)
\\
= &
(\{a_k\}\ast \{\langle k\rangle^{-N}\})(l)\omega_1(l).
\end{split}
\end{equation}
Taking the $l_{p_1}$-norm on both sides of the above inequality, we deduce that
\begin{equation}
\begin{split}
\|C_{\vec{a}, \varphi}\|_{L_{p_1, \omega_1}}
=&
\left\|\left\{\left(\int_{Q_l}|C_{\vec{a}, \varphi}(x)|^{p_1}|\omega_1(x)|^{p_1}dx\right)^{1/{p_1}}\right\}\right\|_{l_{p_1}}
\\
\lesssim &
\|(\{a_k\}\ast \{\langle k\rangle^{-N}\})(l)\omega_1(l)\|_{l_{p_1}}
\\
=&
\|\{a_k\}\ast \{\langle k\rangle^{-N}\}\|_{l_{p_1, \omega_1}}
\\
\lesssim &
\|\vec{a}\|_{l_{p_1, \omega_1}}\cdot \|\{\langle k\rangle^{-N}\}\|_{l_{p_1\wedge 1, v}}
\lesssim
\|\vec{a}\|_{l_{p_1, \omega_1}}.
\end{split}
\end{equation}
We now obtain that $C_{\vec{a}, \varphi}\in L^{\frac{1}{2}Q}_{p_1, \omega_1}$ and $\|C_{\vec{a}, \varphi}\|_{L^{\frac{1}{2}Q}_{p_1, \omega_1}}\lesssim \|\vec{a}\|_{l_{p_1, \omega_1}}$.

On the other hand, recalling that $\vec{a}$ is a nonnegative sequence, $\varphi$ is a nonnegative smooth function with $\varphi(0)>0$,
we can take a small constant $\delta$ such that $\varphi(x)\gtrsim 1$ for $|x|\leq \delta$.
It follows that
\begin{equation}
C_{\vec{a}, \varphi}(x)=\sum_{k\in \mathbb{Z}^n}a_k\varphi(x-k)
\gtrsim a_l
\end{equation}
for $|x-l|\leq \delta$.
Thus
\begin{equation}
\left(\int_{Q_l}|C_{\vec{a}, \varphi}(x)|^{p_2}|\omega_2(x)|^{p_2}dx\right)^{1/{p_2}}
\gtrsim
a_l\omega_2(l).
\end{equation}
Taking the $l_{p_2}$-norm of both sides of the above inequality, we deduce that
\begin{equation}
\begin{split}
\|C_{\vec{a}, \varphi}\|_{L_{p_2, \omega_2}}
=&
\left\|\left\{\left(\int_{Q_l}|C_{\vec{a}, \varphi}(x)|^{p_2}|\omega_2(x)|^{p_2}dx\right)^{1/{p_2}}\right\}\right\|_{l_{p_2}}
\\
\gtrsim &
\|\vec{a}\|_{l_{p_2, \omega_2}}.
\end{split}
\end{equation}
Recalling $C_{\vec{a},\varphi}\in L_{p_1, \omega_1}^{\frac{1}{2}Q}$,
we use the assumption $L^{\Gamma}_{p_1, \omega_1}\subset L_{p_2, \omega_2}$ and Lemma \ref{lemma, independence of Gamma, embedding} to deduce
\begin{equation}
\|C_{\vec{a}, \varphi}\|_{L_{p_2, \omega_2}}
\lesssim
\|C_{\vec{a}, \varphi}\|_{L_{p_1, \omega_1}}.
\end{equation}
Hence
\begin{equation}
\|\vec{a}\|_{l_{p_2, \omega_2}}
\lesssim
\|C_{\vec{a}, \varphi}\|_{L_{p_2, \omega_2}}
\lesssim
\|C_{\vec{a}, \varphi}\|_{L_{p_1, \omega_1}}
\lesssim \|\vec{a}\|_{l_{p_1, \omega_1}}.
\end{equation}
We complete the proof for necessity.
\\
\textbf{Sufficiency.} In this part, we want to verify the embedding relation
$L^{\Gamma}_{p_1, \omega_1}\subset L^{\Gamma}_{p_2, \omega_2}$. By a standard limiting argument, we only need to verify the inequality
\begin{equation}
\|f\|_{L_{p_2, \omega_2}}\lesssim \|f\|_{L_{p_1, \omega_1}}
\end{equation}
for all $f\in \mathscr{S}^{\Gamma}$.
By the spirit of Lemma \ref{lemma, independence of Gamma, embedding}, we take $\Gamma=\frac{1}{2}Q$.

For a fixed $f\in \mathscr{S}^{\Gamma}$, then $\check{f}$ is a smooth function supported in $\frac{1}{2}Q$.
The periodic extension of $\check{f}$, denoted by $P(\check{f})(\xi)=\sum_{l\in \mathbb{Z}^n}\check{f}(\xi+l)$, is defined on $\mathbb{T}^n$. We have the Fourier series
\begin{equation}
P(\check{f})(\xi)=\sum_{k\in \mathbb{Z}^n}f(k)e^{2\pi ik\xi}.
\end{equation}
Taking a smooth cut-off function $\mathscr{F}^{-1}\psi$ with compact support contained in $\frac{2}{3}Q$, such that $\check{\psi}(\xi)=1$ on $\frac{1}{2}Q$, we reconstruct $\check{f}$ by
\begin{equation}
\check{f}(\xi)=\sum_{k\in \mathbb{Z}^n}f(k)e^{2\pi ik\xi}\check{\psi}(\xi).
\end{equation}
We take the Fourier transform of both sides to obtain
\begin{equation}
f(x)=\sum_{k\in \mathbb{Z}^n}f(k)\psi(x-k).
\end{equation}
By the same method in the proof of sufficiency part, we conclude
\begin{equation}
\|f\|_{L_{p_2, \omega_2}}\lesssim \|\{f(k)\}\|_{l_{p_2,\omega_2}}.
\end{equation}

On the other hand, $f(k)$ can be expressed as
\begin{equation}
f(k)=\mathscr{F}(\check{f}\cdot \check{\psi})(k)=\int_{\mathbb{R}}f(y)\psi(k-y)dy.
\end{equation}
For a fixed positive constant $N$, we can find a positive smooth function $\phi$ with compact Fourier support contained in $\frac{1}{2}Q$, such that
\begin{equation}
\phi(x) \gtrsim \langle x\rangle^{-N}.
\end{equation}
In fact, take a nonnegative smooth function $\rho$ \ satisfying $\textbf{Supp}\hat{\rho}\subset \frac{1}{2}Q$ and $\rho(0)>0$.
The function $\phi$ can be chosen as
\begin{equation}
\phi(x)=\int_{\mathbb{R}^n}\rho(x-y)\langle y\rangle^{-N}dy.
\end{equation}
Observing
\begin{equation}
|\psi(k-y)|\lesssim \langle x-y\rangle^{-N}\lesssim  \phi(x-y)
\end{equation}
for $x\in Q_k$,
we obtain that
\begin{equation}
|f(k)|\lesssim \int_{\mathbb{R}}|f(y)||\psi(k-y)|dy\lesssim \int_{\mathbb{R}}|f(y)||\phi(x-y)|dy=(|f|\ast |\phi|)(x)
\end{equation}
for $x\in Q_k$. It then follows
\begin{equation}
|f(k)\omega_1(k)|\lesssim \left(\int_{Q_k}(|f|\ast |\phi|)^{p_1}(x)|\omega_1(x)|^{p_1}dx\right)^{1/{p_1}}.
\end{equation}
Using Young's inequality or Lemma \ref{lemma, convolution for p<1}, we obtain that
\begin{equation}
\begin{split}
\|\{f(k)\}\|_{l_{p_1, \omega_1}}
\lesssim &
\left\|\left\{\left(\int_{Q_k}(|f|\ast |\phi|)^{p_1}(x)|\omega_1(x)|^{p_1}dx\right)^{1/{p_1}}\right\}\right\|_{l_{p_1}}
\\
= &
\left(\int_{\mathbb{R}^n}(|f|\ast |\phi|)^{p_1}(x)|\omega_1(x)|^{p_1}dx\right)^{1/{p_1}}
\\
= &
\||f|\ast |\phi|\|_{L_{p_1, \omega_1}}
\lesssim
\|f\|_{L_{p_1, \omega_1}}\|\phi\|_{L_{p_1\wedge 1, v}}
\lesssim
\|f\|_{L_{p_1, \omega_1}}.
\end{split}
\end{equation}
By the assumption $l_{p_1, \omega_1}\subset l_{p_2, \omega_2}$, we obtain
\begin{equation}
\|f\|_{L_{p_2, \omega_2}}
\lesssim
\|\{f(k)\}\|_{l_{p_2,\omega_2}}
\lesssim
\|\{f(k)\}\|_{l_{p_1,\omega_1}}
\lesssim
\|f\|_{L_{p_1, \omega_1}}
\end{equation}
for all $f\in \mathscr{S}^{\Gamma}$.
We complete the proof for sufficiency.
\end{proof}

\begin{proposition}[\textbf{Second reduction, product}]\label{proposition, second reduction, product}
Suppose $0<p, p_j\leq \infty$, $\omega,\omega_j\in \mathscr{P}({R^n})$ for $j=1,2$.
Let $\Gamma$ be a compact subset of $\mathbb{R}^n$ whose interior  $int(\Gamma)$ is not empty.
Then
\begin{equation}
L^{\Gamma}_{p_1, \omega_1}\cdot L^{\Gamma}_{p_2, \omega_2} \subset L_{p, \omega}
\end{equation}
if and only if
\begin{equation}
l_{p_1, \omega_1}\cdot l_{p_2, \omega_2} \subset l_{p, \omega}.
\end{equation}
\end{proposition}
\begin{proof}
We divide the proof into two parts.
\\
\textbf{Necessity.}
The goal of this part is to verify the relation $l_{p_1, \omega_1}\cdot l_{p_2, \omega_2} \subset l_{p, \omega}$.
By a standard limiting argument, we only need to verify the inequality
\begin{equation}
\|\vec{a}\cdot \vec{b}\|_{l_{p, \omega}}\lesssim \|\vec{a}\|_{l_{p_1, \omega_1}}\|\vec{b}\|_{l_{p_2, \omega_2}}
\end{equation}
for any two nonnegative truncated (only finite nonzero items) sequences $\vec{a}$ and $\vec{b}$.

Let $\varphi$ be a nonnegative smooth function with compact Fourier support contained in $\frac{1}{4}Q$, and satisfying $\varphi(0)>0$.
For two fixed positive truncated sequences $\vec{a}=\{a_k\}_{k\in \mathbb{Z}^n}$ and $\vec{b}=\{b_k\}_{k\in \mathbb{Z}^n}$,
we write
\begin{equation}
C_{\vec{a}, \varphi}(x)=\sum_{k\in \mathbb{Z}^n}a_k\varphi(x-k),~C_{\vec{b}, \varphi}(x)=\sum_{k\in \mathbb{Z}^n}b_k\varphi(x-k).
\end{equation}
As in the proof of Proposition \ref{proposition, second reduction, embedding}, we have
$C_{\vec{a}, \varphi}\in L^{\frac{1}{4}Q}_{p_1, \omega_1}$, $C_{\vec{b}, \varphi}\in L^{\frac{1}{4}Q}_{p_2, \omega_2}$,  and
\begin{equation}
\|C_{\vec{a}, \varphi}\|_{L_{p_1, \omega_1}}\sim \|\vec{a}\|_{l_{p_1, \omega_1}} ,~\|C_{\vec{b}, \varphi}\|_{L_{p_2, \omega_2}}\sim \|\vec{b}\|_{l_{p_2, \omega_2}} .
\end{equation}
In addition, by the truncated property of $\vec{a}$ and $\vec{b}$, both $C_{\vec{a}, \varphi}$ and $C_{\vec{b}, \varphi}$ are Schwartz functions.
Hence, $C_{\vec{a}, \varphi}\cdot C_{\vec{b}, \varphi}$ is a Schwartz function with Fourier support in $\frac{1}{2}Q$.
We use the assumption $L^{\Gamma}_{p_1, \omega_1}\cdot L^{\Gamma}_{p_2, \omega_2} \subset L_{p, \omega}$ to deduce
\begin{equation}
\|C_{\vec{a}, \varphi}\cdot C_{\vec{b}, \varphi}\|_{L_{p, \omega}}
\lesssim
\|C_{\vec{a}, \varphi}\|_{L_{p_1, \omega_1}}\|C_{\vec{b}, \varphi}\|_{L_{p_2, \omega_2}}.
\end{equation}
On the other hand, by the nonnegative of $\varphi$ and the fact $\varphi(0)>0$,
we obtain that
\begin{equation}
(C_{\vec{a}, \varphi}\cdot C_{\vec{b}, \varphi})(x)=\sum_{i, j\in \mathbb{Z}^n}a_i b_j\varphi(x-i)\varphi(x-j)
\gtrsim a_lb_l
\end{equation}
for $|x-l|\leq \delta$.
Thus
\begin{equation}
\left(\int_{Q_l}|(C_{\vec{a}, \varphi}\cdot C_{\vec{b}, \varphi})(x)|^{p}|\omega(x)|^{p}dx\right)^{1/{p}}
\gtrsim
a_lb_l\omega(l).
\end{equation}
So,
\begin{equation}
\begin{split}
\|C_{\vec{a}, \varphi}\cdot C_{\vec{b}, \varphi}\|_{L_{p, \omega}}
=&
\left\|\left\{\left(\int_{Q_l}|(C_{\vec{a}, \varphi}\cdot C_{\vec{b}, \varphi})(x)|^{p}|\omega(x)|^{p}dx\right)^{1/{p}}\right\}\right\|_{l_{p}}
\\
\gtrsim &
\|\{a_lb_l\omega(l)\}\|_{l_{p}}=\|\vec{a}\cdot \vec{b}\|_{l_{p,\omega}}.
\end{split}
\end{equation}
It then yields that
\begin{equation}
\begin{split}
\|\vec{a}\cdot \vec{b}\|_{l_{p,\omega}}
\lesssim &
\|C_{\vec{a}, \varphi}\cdot C_{\vec{b}, \varphi}\|_{L_{p, \omega}}
\\
\lesssim &
\|C_{\vec{a}, \varphi}\|_{L_{p_1, \omega_1}}\|C_{\vec{b}, \varphi}\|_{L_{p_2, \omega_2}}
\\
\lesssim &
\|\vec{a}\|_{l_{p_1,\omega_1}}\|\vec{b}\|_{l_{p_2,\omega_2}}.
\end{split}
\end{equation}
We finish the proof of necessity.
\\
\textbf{Sufficiency.}
In this part, we want to verify the relation $L^{\Gamma}_{p_1, \omega_1}\cdot L^{\Gamma}_{p_2, \omega_2} \subset L_{p, \omega}$.
By the spirit of Lemma \ref{lemma, independence of Gamma, product}, we assume $\Gamma=\frac{1}{4}Q$.

Take $f, g\in \mathscr{S}^{\frac{1}{4}Q}$.
Then
$f\cdot g$ is also a Schwartz function, whose Fourier support is contained in $\frac{1}{2}Q$.
As in the proof of Proposition \ref{proposition, second reduction, product}, we can verify that
\begin{equation}
\|f\|_{L_{p_1,\omega_1}}\sim \|\{f(k)\}\|_{l_{p_1,\omega_1}},~\|g\|_{L_{p_2,\omega_2}}\sim \|\{g(k)\}\|_{l_{p_2,\omega_2}},
\end{equation}
and
\begin{equation}
\|fg\|_{L_{p,\omega}}\sim \|\{f(k)g(k)\}\|_{l_{p,\omega}}.
\end{equation}
Now we use the assumption $l_{p_1, \omega_1}\cdot l_{p_2, \omega_2} \subset l_{p, \omega}$ to deduce
\begin{equation}
\begin{split}
\|fg\|_{L_{p, \omega}}
\sim &
\|\{f(k)g(k)\}\|_{l_{p,\omega}}
\\
\lesssim &
\|\{f(k)\}\|_{l_{p_1,\omega_1}}\|\{g(k)\}\|_{l_{p_2,\omega_2}}
\\
\sim &
\|f\|_{L_{p_1,\omega_1}}\|g\|_{L_{p_2,\omega_2}}.
\end{split}
\end{equation}
We complete the proof of sufficiency.
\end{proof}

\begin{proposition}[\textbf{Second reduction, convolution}]\label{proposition, second reduction, convolution}
Suppose $0<p, p_j\leq \infty$, $\omega, \omega_j\in \mathscr{P}({R^n})$ for $j=1,2$.
Let $\Gamma$ be a compact subset of $\mathbb{R}^n$ whose interior  $int(\Gamma)$ is not empty.
Then
\begin{equation}
L^{\Gamma}_{p_1, \omega_1}\ast L^{\Gamma}_{p_2, \omega_2} \subset L_{p, \omega}
\end{equation}
if and only if
\begin{equation}
l_{p_1, \omega_1}\ast l_{p_2, \omega_2} \subset l_{p, \omega}.
\end{equation}
\end{proposition}
\begin{proof}
We divide the proof into two parts.
\\
\textbf{Necessity.}
The goal of this part is to verify the relation $l_{p_1, \omega_1}\ast l_{p_2, \omega_2} \subset l_{p, \omega}$.
By a standard limiting argument, we only need to verify the inequality
\begin{equation}
\|\vec{a}\ast \vec{b}\|_{l_{p, \omega}}\lesssim \|\vec{a}\|_{l_{p_1, \omega_1}}\|\vec{b}\|_{l_{p_2, \omega_2}}
\end{equation}
for any two nonnegative truncated (only finite nonzero items) sequences $\vec{a}$ and $\vec{b}$.

Let $\varphi$ be a nonnegative smooth function with compact Fourier support contained in $\frac{1}{2}Q$, such that $\varphi(0)>0$.
For two fixed positive truncated sequences $\vec{a}=\{a_k\}_{k\in \mathbb{Z}^n}$ and $\vec{b}=\{b_k\}_{k\in \mathbb{Z}^n}$,
we write
\begin{equation}
C_{\vec{a}, \varphi}(x)=\sum_{k\in \mathbb{Z}^n}a_k\varphi(x-k),~C_{\vec{b}, \varphi}(x)=\sum_{k\in \mathbb{Z}^n}b_k\varphi(x-k).
\end{equation}
As in the proof of Proposition \ref{proposition, second reduction, embedding}, we obtain that
$C_{\vec{a}, \varphi}\in L^{\frac{1}{2}Q}_{p_1, \omega_1}$, $C_{\vec{b}, \varphi}\in L^{\frac{1}{2}Q}_{p_2, \omega_2}$,  and
\begin{equation}
\|C_{\vec{a}, \varphi}\|_{L_{p_1, \omega_1}}\sim \|\vec{a}\|_{l_{p_1, \omega_1}} ,~\|C_{\vec{b}, \varphi}\|_{L_{p_2, \omega_2}}\sim \|\vec{b}\|_{l_{p_2, \omega_2}} .
\end{equation}
In addition, by the truncated property of $\vec{a}$ and $\vec{b}$, both $C_{\vec{a}, \varphi}$ and $C_{\vec{b}, \varphi}$ are Schwartz functions.
Also, $C_{\vec{a}, \varphi}\ast C_{\vec{b}, \varphi}$ is a Schwartz function with Fourier support in $\frac{1}{2}Q$.

We use the assumption $L^{\Gamma}_{p_1, \omega_1}\ast L^{\Gamma}_{p_2, \omega_2} \subset L_{p, \omega}$ to deduce
\begin{equation}
\|C_{\vec{a}, \varphi}\ast C_{\vec{b}, \varphi}\|_{L_{p, \omega}}
\lesssim
\|C_{\vec{a}, \varphi}\|_{L_{p_1, \omega_1}}\|C_{\vec{b}, \varphi}\|_{L_{p_2, \omega_2}}.
\end{equation}
On the other hand, by the positivity of $\varphi$ and the fact $\varphi(0)>0$,
we obtain
\begin{equation}
(C_{\vec{a}, \varphi}\ast C_{\vec{b}, \varphi})(x)=\sum_{i, j\in \mathbb{Z}^n}a_i b_j(\varphi\ast \varphi)(x-i-j)
\gtrsim \sum_{i+j=l}a_ib_j
\end{equation}
for $|x-l|\leq \delta$.
Thus
\begin{equation}
\left(\int_{Q_l}|(C_{\vec{a}, \varphi}\ast C_{\vec{b}, \varphi})(x)|^{p}|\omega(x)|^{p}dx\right)^{1/{p}}
\gtrsim
\sum_{i+j=l}a_ib_j\omega(l).
\end{equation}
So,
\begin{equation}
\begin{split}
\|C_{\vec{a}, \varphi}\ast C_{\vec{b}, \varphi}\|_{L_{p, \omega}}
=&
\left\|\left\{\left(\int_{Q_l}|(C_{\vec{a}, \varphi}\ast C_{\vec{b}, \varphi})(x)|^{p}|\omega(x)|^{p}dx\right)^{1/{p}}\right\}_{l\in \mathbb{Z}^n}\right\|_{l_{p}}
\\
\gtrsim &
\left\|\left\{\sum_{i+j=l}a_ib_j\omega(l)\right\}_{l\in \mathbb{Z}^n}\right\|_{l_{p}}=\|\vec{a}\ast \vec{b}\|_{l_{p,\omega}}.
\end{split}
\end{equation}
We verify
\begin{equation}
\begin{split}
\|\vec{a}\ast \vec{b}\|_{l_{p,\omega}}
\lesssim &
\|C_{\vec{a}, \varphi}\ast C_{\vec{b}, \varphi}\|_{L_{p, \omega}}
\\
\lesssim &
\|C_{\vec{a}, \varphi}\|_{L_{p_1, \omega_1}}\|C_{\vec{b}, \varphi}\|_{L_{p_2, \omega_2}}
\\
\lesssim &
\|\vec{a}\|_{l_{p_1,\omega_1}}\|\vec{b}\|_{l_{p_2,\omega_2}}.
\end{split}
\end{equation}
This finishes the proof of necessity.
\\
\textbf{Sufficiency.}
In this part, we want to verify the relation $L^{\Gamma}_{p_1, \omega_1}\ast L^{\Gamma}_{p_2, \omega_2} \subset L_{p, \omega}$.
By the spirit of Lemma \ref{lemma, independence of Gamma, product}, we assume $\Gamma=\frac{1}{2}Q$.

Take $f, g\in \mathscr{S}^{\frac{1}{2}Q}$. Then
$f\ast g$ is also a Schwartz function, whose Fourier support is contained in $\frac{1}{2}Q$.
As in the proof of Proposition \ref{proposition, second reduction, product}, we can verify that
\begin{equation}
\|f\|_{L_{p_1,\omega_1}}\sim \|\{f(k)\}\|_{l_{p_1,\omega_1}},~\|g\|_{L_{p_2,\omega_2}}\sim \|\{g(k)\}\|_{l_{p_2,\omega_2}}.
\end{equation}
Moreover, we have the following equality:
\begin{equation}
(f\ast g)(k)=\sum_{j+l=k}f(j)g(l).
\end{equation}
As in the proof of Proposition \ref{proposition, second reduction, embedding}, we express the periodizations of $\check{f}$ and $\check{g}$ by their Fourier series
\begin{equation}
P(\check{f})(\xi)=\sum_{k\in \mathbb{Z}^n}f(k)e^{2\pi ik\xi},\ P(\check{g})(\xi)=\sum_{k\in \mathbb{Z}^n}g(k)e^{2\pi ik\xi}.
\end{equation}

A direct calculation now gives that
\begin{equation}
P(\check{f}\check{g})(\xi)=(P(\check{f})P(\check{g}))(\xi)=\sum_{j,l\in \mathbb{Z}^n}f(j)g(l)e^{2\pi i(j+l)\xi}.
\end{equation}
In the above equality, on both sides we multiply a spatial phase $e^{-2\pi ik\xi }$ and take integration over the torus $\mathbb{T}^{n}$.
It then leads to the desired equality\begin{equation}
\begin{split}
\int_{\mathbb{R}^n}\check{f}(\xi)\check{g}(\xi)e^{-2\pi ik\xi} d\xi
=&
\int_{Q}\check{f}(\xi)\check{g}(\xi)e^{-2\pi ik\xi} d\xi
\\
=&
\int_{Q}\sum_{j,l\in \mathbb{Z}^n}f(j)g(l)e^{2\pi i(j+l-k)\xi}d\xi
\\
=&
\sum_{j,l\in \mathbb{Z}^n}\int_{Q}f(j)g(l)e^{2\pi i(j+l-k)\xi}d\xi
\\
=&
\sum_{j+l=k}f(j)g(l).
\end{split}
\end{equation}
By the fact proved in Proposition \ref{proposition, second reduction, embedding}, we conclude
\begin{equation}
\|f\ast g\|_{L_{p,\omega}}\sim \|\{(f\ast g)(k)\}\|_{l_{p,\omega}}=\left\|\left\{\sum_{j+l=k}f(j)g(l)\right\}\right\|_{l_{p,\omega}}.
\end{equation}
Now, we use the assumption $l_{p_1, \omega_1}\ast l_{p_2, \omega_2} \subset l_{p, \omega}$ to deduce
\begin{equation}
\begin{split}
\|f\ast g\|_{L_{p, \omega}}
\sim &
\|\{f(j)\}\ast \{g(l)\}\|_{l_{p,\omega}}
\\
\lesssim &
\|\{f(k)\}\|_{l_{p_1,\omega_1}}\|\{g(k)\}\|_{l_{p_2,\omega_2}}
\\
\sim &
\|f\|_{L_{p_1,\omega_1}}\|g\|_{L_{p_2,\omega_2}}.
\end{split}
\end{equation}
The proof of sufficiency is completed.
\end{proof}

After several reductions, we have finished all the preparations for the proof of our main theorems.
We would like to remark that all the propositions shown in bilinear form in Section 2 and Section 3 can be generalized to the multi-linear form.
Therefore, Theorem \ref{characterization of product, multiple case} can be  verified by Proposition \ref{proposition, first reduction, product} and Proposition \ref{proposition, second reduction, product},
Theorem \ref{characterization of convolution, multiple case} can be verified by Proposition \ref{proposition, first reduction, convolution} and Proposition \ref{proposition, second reduction, convolution},
Theorem \ref{characterization of embedding} can be proved by Proposition \ref{proposition, first reduction, embedding} and Proposition \ref{proposition, second reduction, embedding}.

\begin{remark}
We would like to make a comparison between the known results (for instance, see \cite{Feichtinger_Banach convolution, C.Heil}) and our results.
Recalling that
\begin{equation}
\begin{split}
\|f\|_{\mathcal {M}_{p,q}^{m,\omega}}
=&
\big\|\|V_{\phi}f(x,\xi)\|_{L_{x;p,\omega}}\big\|_{L_{\xi;q,m}}
\\
=&
\big\|\|\mathscr{F}^{-1}(\hat{f}(\cdot)\hat{\phi}(\cdot-\xi))(x)\|_{L_{x;p,\omega}}\big\|_{L_{\xi;q,m}}
\\
=&
\big\|\|\hat{f}(\cdot)\hat{\phi}(\cdot-\xi))\|_{\mathscr{F}(L_{p,\omega})}\big\|_{L_{\xi;q,m}},
\end{split}
\end{equation}
we can re-express the
modulation space by the Wiener type space $W(B, C)$ as in \cite{Feichtinger_Banach convolution} by
\begin{equation}
\mathcal {M}_{p,q}^{m,\omega}=\mathscr{F}^{-1}W(\mathscr{F}L_{p,\omega}, L_{q,m}).
\end{equation}
Then the relation
\begin{equation}
M_{p_1,q_1}^{m_1,\omega_1}\cdot M_{p_2,q_2}^{m_2,\omega_2}\subset M_{p,q}^{m,\omega}
\end{equation}
can be re-expressed as
\begin{equation}
\mathscr{F}^{-1}W(\mathscr{F}L_{p_1,\omega_1}, L_{q_1,m_1})\cdot \mathscr{F}^{-1}W(\mathscr{F}L_{p_2,\omega_2}, L_{q_2,m_2})
\subset \mathscr{F}^{-1}W(\mathscr{F}L_{p,\omega}, L_{q,m})
\end{equation}
which is equivalent to
\begin{equation}\label{for proof, 5}
W(\mathscr{F}L_{p_1,\omega_1}, L_{q_1,m_1})\ast W(\mathscr{F}L_{p_2,\omega_2}, L_{q_2,m_2})
\subset W(\mathscr{F}L_{p,\omega}, L_{q,m}).
\end{equation}
Using Theorem \ref{characterization of product, multiple case}, we know that the above relation (\ref{for proof, 5}) holds if and only if
\begin{equation}
\mathscr{F}L^{\Gamma}_{p_1,\omega_1} \ast \mathscr{F}L^{\Gamma}_{p_2,\omega_2} \subset \mathscr{F}L_{p,\omega},\
L^{\Gamma}_{q_1,m_1}\ast L^{\Gamma}_{q_2,m_2}\subset L_{q,m}
\end{equation}
for some $\Gamma\subset \mathbb{R}^n$ with non-empty interior.

In \cite{Feichtinger_Banach convolution}, the determination of convolution relation
\begin{equation}
W(B_1,C_1)\ast W(B_2,C_2)\subset W(B,C)
\end{equation}
was based on the assumptions on the triplets $(B_1,B_2,B)$ and $(C_1,C_2,C)$,
which are the components of corresponding Wiener amalgam spaces.
However, at least in some cases, our results show that the convolution relation on Wiener type spaces
should be determined by the convolution relation on the local version of their components instead of the full version.
Moreover, in our theorems, one can find that the convolution relation on the local version of the components is weaker than
the convolution relation on the full version of the components. For example, we can verify that
\begin{equation}
L_{q_1,m_1}\ast L_{q_2,m_2}\subset L_{q,m}\Longrightarrow L^{\Gamma}_{q_1,m_1}\ast L^{\Gamma}_{q_2,m_2}\subset L_{q,m},
\end{equation}
but the reverse is not true.
\end{remark}

\section{Applications for more exact weights}
By the spirit of the main theorems, of course we may have many choices to put some assumptions on the index groups $p,p_i,q,q_i,s,s_i$ and on the weights $\omega,\omega_i,m,m_i$,
to assure the embedding relations, product relations, or convolution relations on modulation spaces (or on Wiener amalgam spaces).
However, we do not intend to discuss problems of this type here.
We only focus on the cases that the corresponding characterization can be expressed more concrete. These cases may be more useful for specific applications.

The main theorems established in Section 1 would allow us to reduce the relations about function spaces defined on $\mathbb{R}^n$
to the discrete group $\mathbb{Z}^n$.
Consequently, we have to deal with the corresponding problems on $\mathbb{Z}^n$.
It turns out that working to Lebesgue spaces on $\mathbb{Z}^n$
is more convenient compared to working to the origin function spaces defined on $\mathbb{R}^n$.
This approach may allow us to find sharp conditions for product inequalities, convolution inequalities, embedding relations on function spaces defined on $\mathbb{R}^n$,
and eventually helps us completing some previous studies on this topic.

First, checking the conditions (a) and (b) in Theorem \ref{characterization of product, multiple case} and Theorem \ref{characterization of convolution, multiple case},
we keep the weights on product part and abandon the weights on convolution part,
and obtain the following corollaries corresponding to
Theorem \ref{characterization of product, multiple case} and Theorem \ref{characterization of convolution, multiple case}.
We recall that $0$ means the weight function which is equal to $\langle x\rangle^0=1$ everywhere.

\begin{corollary}
Let $J\geq2$ be an integer. Suppose $1\leq p\leq \infty$, $0<p_j, q, q_j\leq \infty$, $\omega,\omega_j, m, m_j \in \mathscr{P}({R^n})$ for $j=1,2,\cdots ,J$.
Denote $1/r=\max\left\{1/p-\sum_{j=1}^J1/p_j, 0\right\}$, $1/s=\max\left\{1/q-\sum_{j=1}^J1/q_j, 0\right\}$,
$S=\{j\in \mathbb{Z}: \ p_j\geq 1, 1\leq j\leq J.  \}$, $T=\{j\in \mathbb{Z}: \ q_j\geq 1, 1\leq j\leq J.  \}$.
Then
 \begin{enumerate}
   \item
$
\prod_{j=1}^J \mathcal {M}_{p_j,q_j}^{0, \omega_j}\subset \mathcal {M}_{p,q}^{0, \omega}
$
if and only if \
$
\frac{\omega}{\prod_{j=1}^J\omega_j}\in L^r$,  $(|T|-1)+1/q\leq \sum_{j\in T}1/q_j$,  $1/q\leq 1/q_j~(j=1,2,\cdots,J).
$
   \item
$
\coprod_{j=1}^J W_{p_j,q_j}^{m_j, 0}\subset W_{p,q}^{m, 0}
$
if and only if
$
\ (|S|-1)+1/p\leq \sum_{j\in S}1/p_j$,  $1/p\leq 1/p_j~(j=1,2,\cdots, J)$,  $\frac{m}{\prod_{j=1}^J m_j}\in L^s.
$

 \end{enumerate}
\end{corollary}

\begin{corollary}
Let $J\geq2$ be an integer. Suppose $0< p,p_j,q,q_j\leq \infty$, $\omega,\omega_j, m,m_j\in \mathscr{P}({R^n})$ for $j=1,2,\cdots ,J$.
We write $1/r=\max\left\{1/p-\sum_{j=1}^J1/p_j, 0\right\}$, $1/s=\max\left\{1/q-\sum_{j=1}^J1/q_j, 0\right\}$,
$S=\{j\in \mathbb{Z}: \ p_j\geq 1, 1\leq j\leq J.  \}$, $T=\{j\in \mathbb{Z}: \ q_j\geq 1, 1\leq j\leq J.  \}$.
Then
 \begin{enumerate}
   \item
$
\coprod_{j=1}^J \mathcal {M}_{p_j,q_j}^{m_j, 0}\subset \mathcal {M}_{p,q}^{m, 0}
$
if and only if
$
(|S|-1)+1/p\leq \sum_{j\in S}1/p_j$, $1/p\leq 1/p_j~(j=1,2,\cdots, J)$, $\frac{m}{\prod_{j=1}^J m_j}\in L^s.
$
   \item
$
\prod_{j=1}^J W_{p_j,q_j}^{0, \omega_j}\subset W_{p,q}^{0, \omega}
$
if and only if
$
\frac{\omega}{\prod_{j=1}^J \omega_j}\in L^r$, $(|T|-1)+1/q\leq \sum_{j\in T}1/q_j$,  $1/q\leq 1/q_j~(j=1,2,\cdots,J).
$
 \end{enumerate}
\end{corollary}

Also, we have the following corollary from Theorem \ref{characterization of embedding}.

\begin{corollary}
Suppose $0< p_1,p_2,q_1,q_2\leq \infty$, $\omega_j, m_j\in \mathscr{P}({R^n})$ for $j=1,2$.
Denote $1/r=\max\left\{1/p_2-1/p_1, 0\right\}$, $1/s=\max\left\{1/q_2-1/q_1, 0\right\}$.
Then
 \begin{enumerate}
   \item
$
\mathcal {M}_{p_1,q_1}^{m_1,\omega_1}\subset \mathcal {M}_{p_2,q_2}^{m_2,\omega_2}
$
if and only if
$
\omega_2/\omega_1\in L^r,\ m_2/m_1\in L^s.
$
   \item
$
W_{p_1,q_1}^{m_1,\omega_1}\subset W_{p_2,q_2}^{m_2,\omega_2}
$
if and only if
$
\omega_2/\omega_1\in L^r,\ m_2/m_1\in L^s.
$
 \end{enumerate}
\end{corollary}

We reduce the proofs of above three corollaries to the following two lemmas.

\begin{lemma}[Sharpness of H\"{o}lder's inequality, discrete form] \label{lemma, sharpness of the Holder's inequality, discrete form}
Let $J\geq 1$ be an integer. Suppose $0<q,q_j \leq \infty$, $m, m_j\in \mathscr{P}(\mathbb{R}^n)$ for $j=1,2,\cdots ,J$. Then
\begin{equation}\label{for proof, 3}
\prod_{j=1}^Jl_{q_j, m_j}\subset l_{q, m}
\end{equation}
holds if and only if
\begin{equation}
\frac{m}{\prod_{j=1}^J m_j}\in L^r,
\end{equation}
where $1/r=\max\left\{1/q-\sum_{j=1}^J1/q_j, 0\right\}$.
\end{lemma}
\begin{proof}
We first show the necessity of this lemma.
We use $\{\vec{e_j}\}$ to denote the standard orthogonal basis on $\mathbb{Z}^n$, i.e., $\vec{e_j}=\{e_{j,i}\}_{i\in \mathbb{Z}^n}$, $e_{j,i}=1$ for $i=j$, and vanish elsewhere.
For $j\in \mathbb{Z}, 1\leq j\leq J$, we use $\vec{a_j}=\{a_{j,i}\}_{i\in \mathbb{Z}^n}$ to denote a sequence (function) defined on $\mathbb{Z}^n$.

For the case $1/q\leq \sum_{j=1}^J1/q_j$, we take $\vec{a_j}=\vec{e_k}$ for all $1\leq j\leq J$.
Then the inequality
\begin{equation}\label{for proof, 6}
\|\prod_{j=1}^J \vec{a_j}\|_{l_{q,m}}\lesssim  \prod_{j=1}^J \|\vec{a_j}\|_{l_{q_j,m_j}}
\end{equation}
implies that
\begin{equation}
m(k)\lesssim \prod_{j=1}^J m_j(k).
\end{equation}
By the arbitrariness of $k$, we obtain that
\begin{equation}
\frac{m(k)}{\prod_{j=1}^J m_j(k)}\lesssim 1
\end{equation}
for all $k\in \mathbb{Z}^n$,
then we get the conclusion by the fact that $m(\xi)\sim m(k)$ for $\xi\in Q_k$.

For the case $1/q> \sum_{j=1}^J1/q_j$, we take
\begin{equation}
\vec{a_j}=\sum_{|i|\leq N}\frac{1}{m_j(i)}\left(\frac{m(i)}{\prod_{j=1}^Jm_j(i)}\right)^{r/q_j}\vec{e_i}.
\end{equation}


A direct calculation then yields
\begin{equation}
\prod_{j=1}^J\vec{a_j}=\sum_{|i|\leq N}\frac{1}{m(i)}\left(\frac{m(i)}{\prod_{j=1}^Jm_j(i)}\right)^{r/q}\vec{e_i}.
\end{equation}

By the inequality (\ref{for proof, 6}), we obtain that
\begin{equation}
\left(\sum_{|i|\leq N}\left(\frac{m(i)}{\prod_{j=1}^Jm_j(i)}\right)^r\right)^{1/q}
\lesssim
\prod_{j=1}^J \left( \left(\sum_{|i|\leq N}\left(\frac{m(i)}{\prod_{j=1}^Jm_j(i)}\right)^r\right)^{1/q_j}\right)
\end{equation}
which implies
\begin{equation}
\left(\sum_{|i|\leq N}\left(\frac{m(i)}{\prod_{j=1}^Jm_j(i)}\right)^r\right)^{1/r}
\lesssim
1.
\end{equation}
Let $N\rightarrow \infty$, we deduce that $\frac{m}{\prod_{j=1}^J m_j}\in l^r$.
Then the final conclusion will be verified by the fact that $m(\xi)\sim m(k)$ for $\xi\in Q_k$.

For the sufficiency part, (\ref{for proof, 3}) can be verified by a direct application of multi-H\"{o}lder's inequality.
\end{proof}

\begin{lemma}[Sharpness of Young's inequality, discrete form] \label{lemma, sharpness of the Young's inequality, discrete form}
Let $J\geq2$ be an integer. Suppose $0<q,q_j \leq \infty$ for $j=1,2,\cdots ,J$. We write $S=\{j\in \mathbb{Z}: \ q_j\geq 1, 1\leq j\leq J\}$.
Then
\begin{equation}\label{for proof, 4}
\coprod_{j=1}^Jl_{q_j}\subset l_{q}
\end{equation}
holds if and only if
\begin{equation}
\begin{cases}
(|S|-1)+1/q\leq \sum_{j\in S} 1/{q_j},
\\
1/q\leq 1/q_j~(j=1,2\cdots J).
\end{cases}
\end{equation}
\end{lemma}

\begin{proof}
We divide the proof into two parts.\\
$\mathbf{Necessity.}$ We first verify
\begin{equation}\label{for proof, 7}
(|J|-1)+1/q\leq \sum_{j=1}^J 1/{q_j}.
\end{equation}
To this end, for a positive integer $N$, we take
\begin{equation}
\vec{a_j}=\sum_{|i|\leq 2N}\vec{e_i},\ j=1,2,\cdots,J,
\end{equation}
where $\{\vec{e_i}\}_{i\in \mathbb{Z}^n}$ is the standard orthogonal basis on $\mathbb{Z}^n$.
Observing that
\begin{equation}
(\sum_{|i|\leq N}\vec{e_i})\ast (\sum_{|i|\leq 2N}\vec{e_i})
\geq
N^n\sum_{|i|\leq N}\vec{e_i},
\end{equation}
we obtain
\begin{equation}
\coprod_{j=1}^J\vec{a_j} \geq N^{n(|J|-1)}\sum_{|i|\leq N}\vec{e_i}
\end{equation}
by induction.
By the above estimates, we deduce
\begin{equation}
\|\coprod_{j=1}^J\vec{a_j}\|_{l_{q}}
\gtrsim
N^{n(|J|-1+1/q)},
\end{equation}
and
\begin{equation}
\|\vec{a_j}\|_{l_{q_j}}
\sim
N^{n/q_j}.
\end{equation}
By the assumption, we have
\begin{equation}\label{for proof, 8}
\|\coprod_{j=1}^N\vec{a_j}\|_{l_q}\lesssim \prod_{j=1}^N\|\vec{a_j}\|_{l_{q_j}},
\end{equation}
it follows that
\begin{equation}
N^{n(|J|-1+1/q)}\lesssim \prod_{j=1}^J N^{n/q_j}=N^{n(\sum_{j=1}^N1/q_j)}.
\end{equation}
Letting $N\rightarrow \infty$, we obtain (\ref{for proof, 7}).

Secondly, we use some reduction to complete the proof of necessity.
Observe that $\vec{e_0}$ is a unit of convolution on $\mathbb{Z}^n$. More precisely, we have
\begin{equation}
\vec{e_0}\ast \vec{a}=\vec{a}
\end{equation}
for any $\vec{a}$ defined on $\mathbb{Z}^n$.
Then we take $\vec{a_j}=\vec{e_0}$ for $1\leq j\neq i\leq J$, and use (\ref{for proof, 8}) to deduce
\begin{equation}
    \|\vec{a_i}\|_{l_q}\lesssim \|\vec{a_i}\|_{l_{q_i}},
\end{equation}
that is, $l_{q_i}\subset l_q$, which implies $1/q\leq 1/q_i$ by (\ref{for proof, 7}).
By the arbitrary of $i$, we obtain $1/q\leq 1/q_j$ for all $j=1,2\cdots J$.

On the other hand, we take
$\vec{a_j}=\vec{e_0}$ for $j \notin S$, and use (\ref{for proof, 8}) to deduce
\begin{equation}
\|\coprod_{j\in S}\vec{a_j}\|_{l_q}\lesssim \prod_{j\in S}\|\vec{a_j}\|_{l_{q_j}},
\end{equation}
then $(|S|-1)+1/q\leq \sum_{j\in S} 1/{q_j}$ follows by (\ref{for proof, 7}).
Here, we omit the trivial case $S=\emptyset$.
\\
$\mathbf{Sufficiency.}$
The sufficiency for $q\leq 1$ is a direct conclusion of
$$
l_{q_j}\subset l_q\subset l_1, \ \text{and}\  \coprod_{j=1}^Jl_{1}\subset l_{1}.
$$
By the embedding relation $l_q\subset l_1$, we verify that
\begin{equation}
(\coprod_{j=1}^J \vec{a_j})^q\leq \coprod_{j=1}^J (\vec{a_j})^q
\end{equation}
Taking $l_1$-norm of both sides of the inequality, we use classical Young's inequality $\coprod_{j=1}^Jl_{1}\subset l_{1}$ and the embedding relation $l_{q_j}\subset l_q$
to deduce
\begin{equation}
\begin{split}
\|\coprod_{j=1}^J \vec{a_j}\|^q_{l_q}
\lesssim
\|\coprod_{j=1}^J (\vec{a_j})^q\|_{l_1}
\lesssim
\prod_{j=1}^J\|(\vec{a_j})^q\|_{l_1}
=
\prod_{j=1}^J\|\vec{a_j}\|_{l_q}^q
\lesssim
\prod_{j=1}^J\|\vec{a_j}\|_{l_{q_j}}^q,
\end{split}
\end{equation}
which is the desired conclusion.

For the case $q\geq 1$, using the assumptions, we can choose a sequence $\{r_j\}_{j\in S}$ such that $1/r_j\leq 1/q_j$ for $j\in S$,
\begin{equation}
(|S|-1)+1/q= \sum_{j\in S} 1/{r_j},\ q\geq 1,\ r_j\geq 1, j\in S.
\end{equation}
Then we choose $r_j=1$ for $j\notin S$, and deduce
\begin{equation}
(J-1)+1/q= \sum_{j=1}^J 1/{r_j},\ q\geq 1,\ r_j\geq 1, j=1,\cdots, J.
\end{equation}

The classical Young's inequality on $\mathbb{Z}^n$ then implies
$
\coprod_{j=1}^J l_{r_j}\subset l_{q}.
$
Using the embedding relations $l_{q_j}\subset l_{r_j}\ (j=1,\cdots, J)$, we obtain the desired conclusion
\begin{equation}
\coprod_{j=1}^J l_{q_j}\subset l_{q}.
\end{equation}
\end{proof}

Finally, we handle a more special case, i.e., the power weight case.
Since the assumptions of weights in our main theorems in Section 1 are qualitative,
it is hard to give more concrete characterizations under such assumptions (especially for the convolution part).
However, power weight is a more quantitative weight and the relations on modulation spaces and Wiener amalgam spaces
with power weights can be characterized by a more concrete way.

We recall some results obtained in our paper \cite{Guo_SCI China} and \cite{Guo_Sharp convolution}.

\begin{lemma}[See \cite{Guo_SCI China}]\label{lemma, sharpness of weighted Holder's inequality}
Suppose $0<q,q_1,q_2\leq \infty$, $s,s_1,s_2 \in \mathbb{R}$.
Then
\begin{equation}
l_{q_1, s_1}\cdot l_{q_2, s_2} \subset l_{q, s}
\end{equation}
if and only if $(\mathbf{q},\mathbf{s})=(q, q_1, q_2, s, s_1, s_2)$ satisfies one of the following conditions $\mathcal {B}_i$, $i=1,2$.
\begin{eqnarray}
&&(\mathcal {B}_1)\begin{cases}
\frac{1}{q}>\frac{1}{q_1}+\frac{1}{q_2}
\\
\frac{1}{q}+\frac{s}{n}<\frac{1}{q_1}+\frac{s_1}{n}+\frac{1}{q_2}+\frac{s_2}{n};
\end{cases}
\\
&&(\mathcal {B}_2)\begin{cases}
\frac{1}{q}\leq \frac{1}{q_1}+\frac{1}{q_2}
\\
s\leq s_1+s_2.
\end{cases}
\end{eqnarray}
\end{lemma}

\begin{lemma}[See \cite{Guo_Sharp convolution}] \label{lemma, sharpness of weighted Young's inequality}
Suppose $1\leq q, q_1, q_2\leq \infty$, $s, s_1, s_2\in \mathbb{R}$. Then
\begin{equation}
l_{q_1, s_1}\ast l_{q_2, s_2} \subset l_{q, s}
\end{equation}
if and only if $(\mathbf{q},\mathbf{s})=(q, q_1, q_2, s, s_1, s_2)$ satisfies one of the following conditions $\mathcal {A}_i$, $i=1,2,3,4$.
\begin{eqnarray}
&&(\mathcal {A}_1)\begin{cases}
s\leq s_1,~s\leq s_2,~0\leq s_1+s_2,\\
1+\Big(\frac{1}{q}+\frac{s}{n}\Big)\vee 0<\Big(\frac{1}{q_1}+\frac{s_1}{n}\Big)\vee 0+\Big(\frac{1}{q_2}+\frac{s_2}{n}\Big)\vee 0,\\
\frac{1}{q}+\frac{s}{n}\leq \frac{1}{q_1}+\frac{s_1}{n},~\frac{1}{q}+\frac{s}{n}\leq \frac{1}{q_2}+\frac{s_2}{n},
1\leq \frac{1}{q_1}+\frac{s_1}{n}+\frac{1}{q_2}+\frac{s_2}{n},\\
(q,s)=(q_1,s_1) ~\text{if}~ \frac{1}{q}+\frac{s}{n}=\frac{1}{q_1}+\frac{s_1}{n}, \\
(q,s)=(q_2,s_2) ~\text{if}~ \frac{1}{q}+\frac{s}{n}=\frac{1}{q_2}+\frac{s_2}{n},\\
(q'_1, -s_1)=(q_2,s_2)~\text{if}~ 1=\frac{1}{q_1}+\frac{s_1}{n}+\frac{1}{q_2}+\frac{s_2}{n};
\end{cases}
\\
&&(\mathcal {A}_2)\begin{cases}
s=s_1=s_2=0,\\
q=q_1, q_2=1~or~q=q_2, q_1=1~or~q=\infty, \frac{1}{q_1}+\frac{1}{q_2}=1;
\end{cases}
\\
&&(\mathcal {A}_3)\begin{cases}
s\leq s_1,~s\leq s_2,\\
\frac{1}{q_1}+\frac{1}{q_2}=1,~s_1+s_2=0,\\
\frac{1}{q}+\frac{s}{n}<0\leq \frac{1}{q_1}+\frac{s_1}{n},\frac{1}{q_2}+\frac{s_2}{n};
\end{cases}
\end{eqnarray}
\begin{eqnarray}
&&(\mathcal {A}_4)\begin{cases}
s\leq s_1,~s\leq s_2,~0\leq s_1+s_2,\\
1+\frac{1}{q}+\frac{s}{n}=\frac{1}{q_1}+\frac{s_1}{n}+\frac{1}{q_2}+\frac{s_2}{n},~\frac{1}{q}\leq \frac{1}{q_1}+\frac{1}{q_2},\\
\frac{1}{q}+\frac{s}{n}<\frac{1}{q_1}+\frac{s_1}{n},~\frac{1}{q}+\frac{s}{n}<\frac{1}{q_2}+\frac{s_2}{n},~\frac{1}{q}+\frac{s}{n}>0,\\
q\neq \infty,~q_1,q_2\neq 1,~\text{if}~s=s_1~or~s=s_2.
\end{cases}
\end{eqnarray}
Here, we use the notation \
\begin{equation*}
a\vee b=\max \{a,b\}.
\end{equation*}
\end{lemma}

As direct applications of Lemma \ref{lemma, sharpness of weighted Holder's inequality} and Lemma \ref{lemma, sharpness of weighted Young's inequality},
we use Theorem \ref{characterization of product, multiple case} and Theorem \ref{characterization of convolution, multiple case}
to give the characterizations of product inequalities and convolution inequalities on modulation and Wiener amalgam spaces with power weights.

\begin{theorem}[Sharpness of product on weighted modulation spaces] \label{Sharpness of product on modulation spaces}
Suppose $1\leq p, p_j, q, q_j\leq \infty$, $s, s_j, t, t_j \in \mathbb{R}$ for $j=1, 2$. Then
 \begin{enumerate}
   \item
$
M_{p_1, q_1}^{s_1, t_1}\cdot M_{p_2, q_2}^{s_2, t_2}\subset M_{p, q}^{s, t}
$
holds if and only if
$(\mathbf{p},\mathbf{t})=(p, p_1, p_2, t, t_1, t_2)$ satisfies one of the conditions $\mathcal {B}_i$, $i=1,2$
and $(\mathbf{q},\mathbf{s})=(q, q_1, q_2, s, s_1, s_2)$ satisfies one of the conditions $\mathcal {A}_i$, $i=1,2,3,4$.
   \item
$
W_{p_1, q_1}^{s_1, t_1}\ast W_{p_2, q_2}^{s_2, t_2}\subset W_{p, q}^{s, t}
$
holds if and only if
$(\mathbf{p},\mathbf{t})=(p, p_1, p_2, t, t_1, t_2)$ satisfies one of the conditions $\mathcal {A}_i$, $i=1,2$
and $(\mathbf{q},\mathbf{s})=(q, q_1, q_2, s, s_1, s_2)$ satisfies one of the conditions $\mathcal {B}_i$, $i=1,2,3,4$.
 \end{enumerate}
\end{theorem}

\begin{theorem}[Sharpness of convolution on weighted modulation spaces] \label{Sharpness of convolution on modulation spaces}
Suppose $1\leq p, p_j, q, q_j\leq \infty$, $s, s_j, t, t_j \in \mathbb{R}$ for $j=1, 2$. Then
 \begin{enumerate}
   \item
$
M_{p_1, q_1}^{s_1, t_1}\ast M_{p_2, q_2}^{s_2, t_2}\subset M_{p, q}^{s, t}
$
holds if and only if
$(\mathbf{p},\mathbf{t})=(p, p_1, p_2, t, t_1, t_2)$ satisfies one of the conditions $\mathcal {A}_i$, $i=1,2,3,4$
and $(\mathbf{q},\mathbf{s})=(q, q_1, q_2, s, s_1, s_2)$ satisfies one of the conditions $\mathcal {B}_i$, $i=1,2$.
   \item
$
W_{p_1, q_1}^{s_1, t_1}\cdot W_{p_2, q_2}^{s_2, t_2}\subset W_{p, q}^{s, t}
$
holds if and only if
$(\mathbf{p},\mathbf{t})=(p, p_1, p_2, t, t_1, t_2)$ satisfies one of the conditions $\mathcal {B}_i$, $i=1,2,3,4$
and $(\mathbf{q},\mathbf{s})=(q, q_1, q_2, s, s_1, s_2)$ satisfies one of the conditions $\mathcal {A}_i$, $i=1,2$.
 \end{enumerate}
\end{theorem}

We also recall a embedding result about $l_{q,s}$.
This lemma is easy to be verified and hence its proof is not included here.
\begin{lemma}\label{Sharpness of embedding, discrete form}
Suppose $0<q,q_1,q_2\leq \infty$, $s,s_1,s_2\in \mathbb{R}$. Then
\begin{equation}
l_{q_1,s_1}\subset l_{q_2,s_2}
\end{equation}
holds if and only if
\begin{eqnarray}
&&(\mathcal {C}_1)\
\frac{1}{q_1}\leq \frac{1}{q_2},~s_1\leq s_2,
\\
&&(\mathcal {C}_2)\
\frac{1}{q_1}> \frac{1}{q_2},~\frac{1}{q_1}+\frac{s_1}{n}< \frac{1}{q_2}+\frac{s_2}{n}.
\end{eqnarray}

\end{lemma}
With this lemma, combined with Theorem \ref{characterization of embedding},
we obtain the following characterization of embedding between modulation and Wiener amalgam spaces with power weights.

\begin{theorem}[Sharpness of embedding on weighted modulation spaces] \label{Sharpness of embedding on modulation spaces}
Suppose $0\leq p_j, q_j\leq \infty$, $s_j, t_j \in \mathbb{R}$ for $j=1, 2$. Then
 \begin{enumerate}
   \item
$
M_{p_1, q_1}^{s_1, t_1}\subset M_{p_2, q_2}^{s_2, t_2}
$
holds if and only if
$(p_1, p_2, t_1, t_2)$ satisfies one of the conditions $\mathcal {C}_i$, $i=1,2$,
and $(\mathbf{q},\mathbf{s})=(q_1, q_2, s_1, s_2)$ satisfies one of the conditions $\mathcal {C}_i$, $i=1,2$.
   \item
$
W_{p_1, q_1}^{s_1, t_1}\subset W_{p_2, q_2}^{s_2, t_2}
$
holds if and only if
$(p_1, p_2, t_1, t_2)$ satisfies one of the conditions $\mathcal {C}_i$, $i=1,2$,
and $(\mathbf{q},\mathbf{s})=(q_1, q_2, s_1, s_2)$ satisfies one of the conditions $\mathcal {C}_i$, $i=1,2$.
 \end{enumerate}
\end{theorem}

\begin{remark}
As important tools to study some nonlinear problems,
estimates about the product and convolution relations on modulation spaces with power weights have been repeatedly appeared in many articles in the field of PDE.
For example, one can refer to \cite{Iwabuchi, WZG_JFA_2006, WH_JDE_2007}.
So, it is quite interesting to obtain characterizations of the product and convolution relations on modulation spaces.
Such research attracts many authors.
One of the papers that appears very recently is \cite{Toft_sharp convolution}.
However, we find that, in \cite{Toft_sharp convolution}, many endpoint cases are not addressed, which lead gaps between the sufficient and necessary conditions.
As an application of our main theorems, we complete the study in \cite{Toft_sharp convolution} by finding the sharp conditions for the relations on modulation and Wiener amalgam spaces with power weights.
\end{remark}

\subsection*{Acknowledgements}
This work was partially supported by the National Natural Foundation of China (Nos. 11271330 and 11471288).

\end{document}